\documentclass[a4paper,12pt]{amsart}

\usepackage{amsmath}
\usepackage{amssymb}
\usepackage{amsthm}
\usepackage{hyperref}
\usepackage[a4paper]{geometry}
\usepackage{microtype}
\usepackage{todonotes}
\usepackage{enumerate}   
\usepackage{cite}

\DeclareFontFamily{U}{mathx}{\hyphenchar\font45}
\DeclareFontShape{U}{mathx}{m}{n}{
      <5> <6> <7> <8> <9> <10>
      <10.95> <12> <14.4> <17.28> <20.74> <24.88>
      mathx10
      }{}
\DeclareSymbolFont{mathx}{U}{mathx}{m}{n}
\DeclareFontSubstitution{U}{mathx}{m}{n}
\DeclareMathAccent{\widecheck}{0}{mathx}{"71}

\theoremstyle{plain}
\newtheorem{thm}{Theorem}[section]
\newtheorem{cor}{Corollary}[thm]
\newtheorem{lem}[thm]{Lemma}
\newtheorem{prop}[thm]{Proposition}

\theoremstyle{definition}

\newtheorem{exmp}{Example}[section]
\newtheorem*{notation}{Notation}

\theoremstyle{remark}
\newtheorem*{rem}{Remark}

\newcommand\inner[2]{\langle #1, #2 \rangle}
\newcommand{\parcheck}{\check{\phantom{x}}}		%Upside-down hat for longer expressions

\newcommand{\tco}{\mathcal{T}^1}
\newcommand{\bo}{B(L^2(\reals^d))}
\newcommand{\reals}{\mathbb{R}}

\newcommand{\HS}{L^2(\reals^d)}
\newcommand{\F}{\mathcal{F}}
\newcommand{\SC}{\mathcal{T}}
\newcommand{\tr}{\mathrm{tr}}

\DeclareMathOperator{\spct}{sp}

\begin{document}
%%%%%%%%%%%%%%%%%%%%%%%%%%%%%%%%%%%%%%%%%%%%%%%%%%%%%%%%%%%%%%%%%%%%%%%%
\pagestyle{plain}
\title{Convolutions for localization operators}
\author{Franz Luef}
\author{Eirik Skrettingland} 
%\thanks{The author acknowledges support from the Austrian Academy of Sciences as an APART fellow.}
\address{Department of Mathematics\\ NTNU Norwegian University of Science and
Technology\\ NO–7491 Trondheim\\Norway}
\email{franz.luef@math.ntnu.no, eirik.skrettingland@ntnu.no}
%  Math Subject Classifications
%%%%%%%%%%%%%%%%%%%%%%%%%%%%%%%%%%%%%%%%%%%%%%%%%%%%%%%%%%%%%%%%%%%%%%%%%%%%%%%%%%%%%%%%%%%%%%%%%%%%%%%%%%%%%%%%%%%%%%%%%%%%%%%%%%%%%%%%%%%%%%%%%%%%%%%
\begin{abstract}
Quantum harmonic analysis on phase space is shown to be linked with localization operators. The convolution between operators and the convolution 
between a function and an operator provide a conceptual framework for the theory of localization operators which is complemented by an appropriate 
Fourier transform, the Fourier-Wigner transform. We link the Hausdorff-Young inequality for the Fourier-Wigner transform with Lieb's inequality for ambiguity functions. Noncommutative Tauberian theorems due to Werner allow us to extend results of Bayer and 
Gr\"ochenig on localization operators. Furthermore we show that the Arveson spectrum and the theory of Banach modules provide the abstract setting of quantum harmonic analysis.    
\end{abstract}
\maketitle \pagestyle{myheadings} \markboth{F. Luef and E. Skrettingland}{Convolutions for localization operators}
\thispagestyle{empty}
%%%%%%%%%%%%%%%%%%%%%%%%%%%%%%%%%%%%%%%%%%%%%%%%%%%%%%%%%%%%%%%%%%%%%%%%%%%%%%%%%%%%%%%%%%%%%%%%%%%%%%%%%%%%%%%%%%%%%%%%%%%%%%%%%%%%%%%%%%%%%%%%%%%%%%%%
\section{Introduction}

Localization operators are operators of the form

\[ \mathcal{A}_f^{\varphi_1,\varphi_2}\phi=\int_{\mathbb{R}^{2n}} f(z)\langle\phi,\pi(z)\varphi_1\rangle\pi(z)\varphi_2 \,dz,\]
for some window functions $\varphi_1,\varphi_2$,  $\pi(z)\phi(t)=e^{2\pi i\omega t}\phi(t-x)$  and a mask $f$. Boundedness properties on 
modulation spaces and other function spaces have received some attention during the last years, see for example \cite{Bayer:2014td,Boggiatto:2004,Cordero:2003ke,coni16}. 

We contribute a conceptual approach to these localization operators based on Werner's theory of quantum harmonic analysis on phase space \cite{Werner:1984}. Our presentation of \cite{Werner:1984} is based on terminology and notation from time-frequency analysis and harmonic analysis. Various proofs in \cite{Werner:1984} are just indicated and we have decided to include complete proofs for all the major results in Werner's theory. We hope that in this way the deep results by Werner become more accessible to a wider audience. 

Convolutions of functions with operators and of operators with operators have been introduced in \cite{Werner:1984}, along with a corresponding Fourier transform of operators -- the Fourier-Wigner transform. The convolution is based on a natural notion of translation of an operator via conjugation of an operator by the Schr\"odinger representation. Localization operators and the Berezin transform are expressed in terms of convolutions between a function and an operator, and between two operators. The Fourier-Wigner transform behaves like the one for functions as it maps convolutions into products, there is a Riemann-Lebesgue lemma and a Hausdorff-Young inequality \cite{Werner:1984}. We complement Werner's results with the observation that the Hausdorff-Young inequality in the rank-one case yields Lieb's uncertainty principle with constant one \cite{li90-1}. Hence if one invokes Lieb's inequality in the proof of the Hausdorff-Young inequality for the Fourier-Wigner transform and the singular value decomposition, then one obtains a sharp inequality for trace class operators in certain cases. 

Using these concepts we formulate and prove a version of Wiener`s Tauberian theorem for operators due to Werner. These variants of Tauberian theorems have shown to be of 
relevance in quantum mechanics and quantum information theory \cite{Keyl:2015bn,Kiukas:2012gt}. 

The main novel contribution is a formulation of localization operators using the convolution of a function with a rank-one operator, which gives an extension of results in 
\cite{Bayer:2014td} on the following problem: what conditions must be imposed on the windows $\varphi_1,\varphi_2$ to guarantee that the set $\{A_f^{\varphi_1,\varphi_2}|\,f\in L^1(\mathbb{R}^{2d})\}$ is dense in different spaces of operators?

In addition, the theory of Banach modules is used to prove new results on the convolutions, and the Fourier-Wigner transform is shown to be related to the 
Arveson spectrum. Finally the convolutions are considered in the context of modulation spaces, inspired by the existing literature on localization operators 
and modulation spaces.

\section{Prerequisites}
\subsection{Notation and conventions} \label{sec:notation}
If $X$ is a Banach space we will denote its dual space by $X^*$, and for $x\in X$ and $x^* \in X^*$ we write $\inner{x^*}{x}$ to denote $x^*(x)$. In order to agree with inner product notation, the duality bracket $\inner{\cdot}{\cdot}$ will always be antilinear in the second argument. 
 
Elements of $\reals^{2d}$ will often be written in the form $z=(x,\omega)$ for $x,\omega\in \reals^d$. Functions on $\reals^d$ and $\reals^{2d}$ will often play different roles, so we denote functions on $\reals^{d}$ by Greek letters such as $\psi,\phi$, and functions on $\reals^{2d}$ by Latin letters such as $f,g$. If $\psi$ and $\phi$ are functions on $\reals^d$, then we write $\psi \otimes \phi$ for the function on $\reals^{2d}$ defined by $\psi \otimes \phi(x,\omega)=\psi(x)\phi(\omega)$. Similarly, for two elements $\xi, \eta$ in some Hilbert space $\mathcal{H}$, we define the operator $\xi \otimes \eta$ on $\mathcal{H}$ by $\xi \otimes \eta (\zeta)=\inner{\zeta}{\eta}\xi$, where $\zeta \in \mathcal{H}$ and $\inner{\cdot}{\cdot}$ is the inner product on $\mathcal{H}$. 
The class of Schwartz functions on $\reals^d$ will be denoted by $\mathcal{S}(\reals^d)$, and the space of tempered distributions by $\mathcal{S}'(\reals^d)$. 

Given a function $\psi:\reals^d\to \mathbb{C}$, we define $\psi^*$ by $\psi^*(x)=\overline{\psi(x)}$ for any $x\in \reals^d$. Similarly we introduce the notation $\check{\psi}$ along with the parity operator $P$ by  $\check{\psi}(x)=P\psi(x)=\psi(-x)$ for any $x\in \reals^d$.

For $p<\infty$, $\SC^p(\mathcal{H})$ will denote the Schatten $p$-class of operators on a Hilbert space $\mathcal{H}$ with singular values in $\ell^p$. We will just write $\SC^p$ when $\mathcal{H}=\HS$. Furthermore, we define $\SC^{\infty}(\mathcal{H})$ to be $B(\mathcal{H})$; all the bounded, linear operators on $\mathcal{H}$. $K(\mathcal{H})$ denotes the closed ideal of compact operators.  

\subsection{Schatten $p$-classes}
For $T\in \tco(\mathcal{H})$ we define the \textit{trace} of $T$ by $\tr(T)=\sum\limits_{n\in \mathbb{N}} \inner{Te_n}{e_n}$, where $\{e_n\}_{n\in \mathbb{N}}$ is some orthonormal basis for $\mathcal{H}$. The trace is linear and independent of the orthonormal basis used to calculate it \cite{Conway:2000uq}. The next proposition, mainly from \cite[Thm. 18.11]{Conway:2000uq}, collects some standard properties of the trace that we are going to need later. Note that part \ref{prop:traceprops:absconv} is merely included as a step in the proof of part \ref{prop:traceprops:abstrace} in \cite{Conway:2000uq}, but we will find it useful one some occasions. 
\begin{prop}
\label{prop:traceprops}
	Let $S\in \tco(\mathcal{H})$, $A\in B(\mathcal{H})$.
	\begin{enumerate}
		\item $S^*\in \tco(\mathcal{H})$, and $\tr(S^*)=\overline{\tr(S)}$.
		\item $\tr(AS)=\tr(SA)$.
		\item \label{prop:traceprops:absconv} $\sum\limits_{n \in \mathbb{N}} |\inner{ASe_n}{e_n}|\leq \|A\|_{B(L^2)} \|S\|_{\tco}$.
		\item \label{prop:traceprops:abstrace} $|\tr(AS)|\leq \|A\|_{B(L^2)} \|S\|_{\SC^1}$.
	\end{enumerate}
\end{prop}

We now state the duality relations of Schatten $p$-classes \cite{Simon:2010wc}. Note that we adhere to our convention that the duality bracket is antilinear in the second argument. 
\begin{lem}
For $1<p<\infty$ the dual space of $\SC^p(\mathcal{H})$ is $\SC^{q}(\mathcal{H})$, and the duality may be given by 
		\begin{equation*}
  		\inner{T}{S}=\tr(TS^*)
		\end{equation*}
		for $S \in \SC^p(\mathcal{H})$ and $T\in \SC^{q}(\mathcal{H})$.
Furthermore, the dual space of $\tco(\mathcal{H})$ is $B(\mathcal{H})$ and the dual space of $K(\mathcal{H})$ is $\tco(\mathcal{H})$ under the same duality action. 
\end{lem}

\subsection{Vector-valued integration} \label{sec:integration}
Given a Banach space $X$, a locally compact group $G$ and a function $F:G\to X$, we say that $F$ is \textit{integrable} if $\phi \circ F:G \to \mathbb{C}$ is integrable with respect to Haar measure for any bounded linear functional $\phi$ on $X$. The \textit{integral} $\int_{G} F \ d\mu$ of $F$,  where $\mu$ is Haar measure, is then a vector $v\in X$ such that $\phi(v)=\int_{G} \phi \circ F \ d\mu$ for any bounded linear functional $\phi$ on $X$. We now cite a sufficient condition for the existence of the integral, which is \cite[Thm. A.22]{Folland:2015hg}. 
\begin{thm}
\label{thm:intexists}
	Let $X$ be a Banach space, $\mu$ Haar measure on $G$, $g:G \to \mathbb{C}$ a function in $L^1(G)$ and $F:G \to X$ a bounded and continuous function. In this case the integral $\int_{G} g\cdot F \ d\mu$ exists in the sense discussed above, belongs to the closed linear span of the range of $F$ and satisfies the norm estimate
	\begin{equation*}
  	\left\Vert\ \int_{G} gF \ d\mu\right\Vert_X \leq \|g\|_{L^1(G)} \sup_{x\in G} \|F(x)\|_X .
	\end{equation*}
\end{thm}
Since the definition of the integral is based on the dual space $X^*$, this definition is often called the "weak" definition of the integral \cite{Folland:2015hg}. By definition, the integral commutes with linear functionals. As is shown in \cite{Folland:2015hg}, it actually commutes with all bounded, linear operators.
\begin{prop}
\label{prop:integralandoperator}
	Let $X,Y$ be Banach spaces, $A:X\to Y$ a bounded linear operator and $\mu$ Haar measure. If $F:G \to X$ is an integrable function such that the integral $\int_{G} F \ d\mu$ exists in $X$, then $T\circ F$ is an integrable function such that the integral $\int_{G} T \circ F \ d\mu$ exists in $Y$. Furthermore, $\int_{G} T \circ F \ d\mu=T \left(  \int_{G} F \ d\mu \right)$.
\end{prop}

In this paper we will consider $F:G\to B(\mathcal{H})$ that are only strongly continuous, where $\mathcal{H}$ is a Hilbert space. Theorem \ref{thm:intexists} is therefore not directly applicable, and we need to define the integral of $F$ pointwise. If we fix $\eta \in \mathcal{H}$, then the map $F_{\eta}:G \to \mathcal{H}$ given by $F_{\eta}(t)=F(t)(\eta)$ is continuous, and theorem \ref{thm:intexists} tells us that $\int_{G} g\cdot F_{\eta} \ d\mu$ exists as an element of $\mathcal{H}$ for $g\in L^1(G)$. It is then a simple exercise to show that $\eta \mapsto \int_{G} g\cdot F_{\eta} \ d\mu$ is a bounded, linear operator on $\mathcal{H}$, and it is this operator that we will denote by $\int_{G} g\cdot F \ d\mu$ in this case. It is then true that $\|\int_{G} g\cdot F \ d\mu\|_{B(\mathcal{H})}\leq \|g\|_{L^1} \sup_{z\in G} \|F(z)\|_{B(\mathcal{H})}$ \cite{Folland:2015hg}.
\begin{prop}
\label{prop:intistrace}
	Let $\mathcal{H}$ be a Hilbert space, $G$ a locally compact group, $U:G \to \mathcal{U}(\mathcal{H})$ a strongly continuous function, $T\in \tco(\mathcal{H})$, and $f\in L^1(G)$. Here $\mathcal{U}(\mathcal{H})$ denotes the unitary operators. Define the operator $I_T$ by
\begin{equation*}
  I_T=\iint_{G} f(z) U(z)TU(z)^* \ dz.
\end{equation*}
$I_T\in \tco(\mathcal{H})$ with $\|I_T\|_{\tco}=\|f\|_{L^1} \|T\|_{\tco}$, and if $S\in B(\mathcal{H})$ then
\begin{equation*}
  \tr(SI_T)=\iint_{G} f(z)\tr(S U(z)TU(z)^*) \ dz.
\end{equation*}

\end{prop}

\begin{proof}
	The strong continuity of $z\mapsto U(z)TU(z)^*$ follows from the strong continuity of $U(z)$, so the integral defining $I_T$ exists by the preceding discussion. 
	
	A slightly tedious but straightforward calculation using proposition \ref{prop:integralandoperator} confirms that  $| I_T|=\iint_{G} |f(z)| U(z)|T|U(z)^* \ dz$, where $|T|$ denotes the positive part $(T^*T)^{1/2}$ in the polar decomposition of $T$. If $\{e_n\}_{n\in \mathbb{N}}$ is an orthonormal basis for $\mathcal{H}$, the trace class norm $\tr(|I_T|)$ is given by
	\begin{align*}
  	\sum_{n \in \mathbb{N}} \inner{\iint_{G} |f(z)| U(z)|T|U(z)^* \ dz \ e_n}{e_n} &=\sum_{n\in \mathbb{N}}\iint_{G} |f(z)| \inner{U(z)|T|U(z)^*e_n}{e_n} \ dz \\
  	&= \iint_{G} |f(z)| \sum_{n\in \mathbb{N}} \inner{|T|U(z)^*e_n}{U(z)^*e_n} \ dz \\
  	&= \|T\|_{\tco} \|f\|_{L^1}.
	\end{align*}
	We have used that $\{U(z)^*\psi_n\}_{n\in \mathbb{N}}$ is another orthonormal basis since $U(z)$ is unitary, and the trace is independent of the basis used to calculate it. We have also used Tonelli's theorem to switch the order of the sum and integral.
	
	In order to prove the last formula, let $\{e_n\}_{n\in \mathbb{N}}$ be an orthonormal basis for $\mathcal{H}$.
	\begin{align*}
  \tr(SI_T)&= \sum_{n\in \mathbb{N}} \inner{SI_Te_n}{e_n} \\
  &= \sum_{n\in \mathbb{N}} \iint_{G} f(z) \inner{SU(z)TU(z)^* e_n}{e_n} \ dz,
\end{align*}
where we have used proposition \ref{prop:integralandoperator} to move $S$ inside the integral, and then moved the inner product inside the integral by the weak definition of the integral. The result would clearly follow if we could move the sum inside the integral, and we therefore use Fubini's theorem which applies since 
\begin{align*}
  \iint_{G} \sum_{n\in \mathbb{N}}|f(z) \inner{SU(z)TU(z)^* e_n}{e_n}| \ dz \leq \|T\|_{\tco} \|S\|_{B(\mathcal{H})} \iint_{G} |f(z)|  \ dz < \infty
\end{align*}
by part \eqref{prop:traceprops:absconv} of proposition \ref{prop:traceprops}.
\end{proof}

\subsection{Modulation spaces} \label{sec:modspaces} 
The modulation spaces are a class of spaces of functions and distributions introduced by Feichtinger in a series of papers starting with the introduction of the so-called Feichtinger algebra in \cite{Feichtinger:1981fz}, and they have since been recognized as a suitable setting for time-frequency analysis \cite{Grochenig:2001}. To define these spaces, we need to define the fundamental operators in time-frequency analysis. If $\psi:\reals^d\to \mathbb{C}$ and $z=(x,\omega)\in \reals^{2d}$, we define the \textit{translation operator} $T_x$ by $T_x\psi (t)=\psi(t-x)$, the \textit{modulation operator} $M_{\omega}$ by $M_{\omega}\psi (t)=e^{2 \pi i \omega \cdot t} \psi (t)$ and the \textit{time-frequency shifts} $\pi(z)$ by $\pi(z)=M_{\omega}T_x$. The translation and modulation operators satisfy the important commutation relation $M_{\omega}T_x = e^{2\pi i x \cdot \omega} T_x M_{\omega}$.

For $\psi,\phi \in L^2(\reals^d)$ the \textit{short-time Fourier transform} (STFT) $V_{\phi}\psi$ of $\psi$ with window $\phi$ is the function on $\reals^{2d}$
 defined by
 \begin{equation*}
  V_{\phi}\psi(z)=\inner{\psi}{\pi(z)\phi}
\end{equation*}
for $z\in \reals^{2d}$. The STFT can be extended to other spaces by interpreting the bracket $\inner{\cdot}{\cdot}$ as a duality bracket. This allows us to consider $V_{\phi}\psi$ for $\phi \in \mathcal{S}(\reals^d)$ and $\psi \in \mathcal{S}'(\reals^d)$. 
We further define the \textit{cross-ambiguity function} $A(\psi,\phi)$ of $\psi$ and $\phi$ by 
\begin{equation*}
  A(\psi, \phi)(z)=e^{\pi i x \cdot \omega} V_{\phi}\psi(z).
\end{equation*}

To define the modulation spaces we fix a window $\phi \in \mathcal{S}(\reals^d)\setminus \{0\}$. For $1\leq p,q\leq \infty$, the \textit{modulation space} $M^{p,q}(\reals^d)$ is the set of tempered distributions $\psi$ such that 
	\begin{equation*}
 \|\psi \|_{M^{p,q}} = \left( \int_{\reals^d} \left(\int_{\reals^d} |V_{\phi}\psi(x,\omega)|^p \ dx \right)^{q/p} \ d\omega \right)^{1/q} < \infty.
 \end{equation*}
 In the special cases where $p$ or $q$ is $\infty$, the integral is replaced by an essential supremum. When $p=q$, we will denote the space $M^{p,p}(\reals^{d})$ by $M^p(\reals^d)$.

 The modulation spaces are Banach spaces with the norms $\|\psi\|_{M^{p,q}}$, and using a different window $\phi \in \mathcal{S}(\reals^d)\setminus \{0\}$ in the definition yields the same spaces with equivalent norms \cite{Grochenig:2001}. 

\begin{lem}[Moyal's identity]
If $\psi_1, \psi_2, \phi_1, \phi_2 \in L^2(\reals^d)$, then $V_{\phi_i}\psi_j \in L^2(\reals^{2d})$ for $i,j\in \{1,2\}$, and the relation
\begin{equation*}
	\inner{V_{\phi_1}\psi_1}{V_{\phi_2}\psi_2}=\inner{\psi_1}{\psi_2}\overline{\inner{\phi_1}{\phi_2}}
\end{equation*}
holds, where the leftmost inner product is in $L^2(\reals^{2d})$ and those on the right are in $\HS$.
\end{lem}
The following proposition is sometimes known as Lieb's uncertainty principle \cite{Grochenig:2001}.
\begin{prop} \label{prop:lieb}
	Let $\phi,\psi\in \HS$, and let $2\leq p < \infty$. Then
\begin{equation*}
  \iint_{\reals^{2d}} |V_{\phi}\psi(z)|^p \ dz \leq \left(\frac{2}{p}\right)^d \|\phi\|_{L^2}^p\|\psi\|_{L^2}^p.
\end{equation*}
\end{prop}

\subsubsection{Wilson bases} \label{sec:wilson}
A very useful property of the modulation spaces $M^p(\reals)$ is the existence of a so-called \textit{Wilson basis} $\mathcal{W}(g)=\{\psi_{k,n}\}_{k\in \mathbb{Z},n\geq 0}$, where $g\in L^2(\reals)$.  We will not discuss the details of this construction, but confine ourselves with knowing that there exists a Wilson basis $\mathcal{W}(g)=\{\psi_{k,n}\}_{k\in \mathbb{Z},n\geq 0}$, that is an orthonormal basis of $L^2(\reals)$ as well as an unconditional basis for $M^p(\reals)$ for $1\leq p < \infty$ \cite{Grochenig:2001}. Furthermore, for every $\phi\in M^1(\reals)$, the expansion 
\begin{equation*}
  \phi=\sum_{k\in \mathbb{Z},n\geq 0} \inner{\phi}{\psi_{k,n}} \psi_{k,n}
\end{equation*}
 converges unconditionally in the norm of $M^1(\reals)$, and the expression $\|\phi\|=\sum_{k,n} |\inner{\phi}{\psi_{k,n}}|$ is a norm on $M^1(\reals)$, equivalent to the usual one \cite{Feichtinger:1992, Heil:2008uo}. A Wilson basis with the same properties for $M^1(\reals^d)$ is obtained by taking tensor products. For instance, if $\{\psi_{k,n}\}_{k\in \mathbb{Z},n\geq 0}$ is a Wilson basis for $M^1(\reals^d)$, then $\{\psi_{k,n}\otimes \psi_{i,j}\}_{k,i\in \mathbb{Z},n,j\geq 0}$ is a Wilson basis for $M^1(\reals^{2d})$.

\subsection{The symplectic Fourier transform} 
The standard symplectic form $\sigma$ is defined for $(x_1,\omega_1), (x_2,\omega_2)\in \reals^{2d}$ by $\sigma(x_1,\omega_1;x_2, \omega_2)=\omega_1\cdot x_2-\omega_2 \cdot x_1$. Using the standard symplectic form we can introduce a version of the Fourier transform that will be suitable for the consideration in this text. For $f\in L^1(\reals^{2d})$ we define the \emph{symplectic Fourier transform} $\F_{\sigma}f$ of $f$ to be the function
\begin{equation*}
\F_{\sigma}f(z)=\iint_{\reals^{2d}} f(z') e^{-2 \pi i\sigma(z,z')} \ dz'
\end{equation*}
for $z\in \reals^{2d}$, where $\sigma$ is the standard symplectic form.

The symplectic Fourier transform and the regular Fourier transform $\F f(z)=\iint_{\reals^{2d}} f(z') e^{-2\pi i z\cdot z'} \ dz'$ are related by $\F_{\sigma}f(x,\omega)=\F f(\omega,-x)$. From this it follows that most properties of the Fourier transform carry over to the symplectic version, and that $F_{\sigma}$ is its own inverse on $\HS$ \cite{deGosson:2011wq}.

\subsection{Pseudodifferential operators}
This section will introduce three procedures for associating a bounded operator on $\HS$ to functions on $\reals^{2d}$, or more generally to distributions in $\mathcal{S}'(\reals^d)$.  The procedures come with different formalisms and properties that we will take advantage of, that any continuous operator $A:\mathcal{S} \to \mathcal{S}'$ may be expressed in all of the three representations under consideration, by \cite[Thm. 14.3.5]{Grochenig:2001}. 

\subsubsection{The Weyl calculus}
A close relative of the STFT is the \textit{cross-Wigner distribution} of two functions $\psi$ and $\phi$ on $\reals^d$. By definition, the cross-Wigner distribution $W(\psi,\phi)$ is given by
\begin{equation*}
  W(\psi,\phi)(x,\omega)=\int_{\reals^d} \psi\left(x+\frac{t}{2}\right)\overline{\phi\left(x-\frac{t}{2}\right)} e^{-2 \pi i \omega \cdot t} \ dt.
 \end{equation*}	
This expression is similar to the definition of the STFT and the cross-ambiguity function, and in fact $W(\psi,\phi)=\F_{\sigma}A(\psi,\phi)$ \cite{deGosson:2011wq}.

Our main motivation for studying the cross-Wigner distribution is its connection with the \textit{Weyl calculus}. For $\sigma\in \mathcal{S}'(\reals^{2d})$ and $\psi,\phi \in \mathcal{S}(\reals^d)$, we define the \textit{Weyl transform} $L_{\sigma}$ of $\sigma$ to be the operator given by 
\begin{equation*}
  \inner{L_{\sigma}\psi}{\phi}=\inner{\sigma}{W(\phi,\psi)}.
\end{equation*}

$\sigma$ is called the \textit{Weyl symbol} of the operator $L_{\sigma}$.
	
\subsubsection{The integrated Schr{\"o}dinger representation and twisted convolution} \label{sec:intschrodrep}
Another way of associating an operator to a function is to define the operator as a superposition of time-frequency shifts using the theory of vector-valued integration. The \textit{integrated Schr{\"o}dinger representation} is the map $\rho: L^1(\reals^{2d})\to \bo$ given by 
	\begin{equation*}
  \rho(f)=\iint_{\reals^{2d}} f(z) e^{-\pi i x \cdot \omega} \pi(z) \ dz,
\end{equation*}
where the integral is defined in the weak and pointwise sense discussed in section \ref{sec:integration}. We say that $f$ is the \textit{twisted Weyl symbol} of $\rho(f)$.

We will use the important product formula $\rho(f)\rho(g)=\rho(f \natural g)$, where the product $\natural$ is the \textit{twisted convolution}, defined by 
$
  f\natural g (z) = \iint_{\reals^{2d}} f(z-z')g(z') e^{\pi i \sigma(z,z')} \ dz'
$
for $f,g \in L^1(\reals^{2d})$  \cite{Folland:1989tl,Grochenig:2001} .

For this paper it is essential that $\rho$ may be extended to a unitary operator from $L^2(\reals^{2d})$ to $\SC^2$, and that the twisted convolution  $f\natural g$ may be defined for $f,g \in L^2(\reals^{2d})$ with norm estimate $\|f\natural g\|_{L^2}\leq \|f\|_{L^2}\|g\|_{L^2}$. Both of these facts are proved in \cite{Folland:1989tl}, in theorem 1.30 and proposition 1.33, respectively.

The relationship between the Weyl calculus and the integrated Schr{\"o}dinger representation is given by  $L_{f}=\rho(\F_{\sigma}f)$ for a symbol $f$.

\subsubsection{Integral operators} \label{sec:intoperators}

Finally one may assign to a function $k$ on $\reals^{2d}$ a so-called \textit{integral operator} $A_k$ on $\HS$ by 
\begin{equation} \label{eqn:integraloperator}
  A_k\psi(s)=\int_{\reals^{d}} k(s,t) \psi(t) \ dy
\end{equation}
 for $\psi \in \HS$.  $k$ is called the \textit{kernel} of $A_k$.  
 \begin{notation}
	We will let $\mathcal{M}$ denote the set of integral operators $A_k$ with kernel $k$ in $M^1(\reals^{2d})$.
\end{notation}
As is shown in \cite{Heil:2008uo}, $\mathcal{M}$ is also the set of operators with Weyl symbol or twisted Weyl symbol in $M^1(\reals^{2d})$. The next theorem (see \cite{Heil:2008uo}) shows that operators in $\mathcal{M}$ have a useful decomposition in terms of the Wilson basis. 

 \begin{thm} \label{thm:integraloperators}
 Let $k\in M^1(\reals^{2d})$ and let $A_k$ be the integral operator with kernel $k$. Let $\{w_n\}_{n\in \mathbb{N}}$ be a Wilson basis for $\HS$, and denote by $W_{mn}$ the corresponding Wilson basis for $L^2(\reals^{2d})$ given by $W_{mn}(x,y)=w_m(x)\overline{w_n(y)}$. 
 
 Then $A_k \in \tco$ with $\|A_k\|_{\tco}\leq K \|k\|_{M^1}$ for some constant $K$, and $A_k=\sum_{m,n \in \mathbb{N}} \inner{k}{W_{mn}}w_m\otimes w_n$ where the sum converges in the $\tco$- norm.
 \end{thm}

\subsection{Localization operators and Berezin transform}
Let $\varphi_1$ and $\varphi_2$ be two functions on $\reals^d$, called \textit{windows}. If $f$ is a function on $\reals^{2d}$, then the \textit{localization operator with symbol} $f$ is the operator $A_f^{\varphi_1,\varphi_2}$ on $\HS$ defined by
\begin{equation*}
  A_f^{\varphi_1,\varphi_2}\psi=\iint_{\reals^{2d}} f(z) \cdot V_{\varphi_1}\psi(z) \pi(z) \varphi_2 \ dz
\end{equation*}
for $\psi\in \HS$. The integral is interpreted in the weak sense discussed in section $\ref{sec:integration}$.

If $T \in \bo$, the \textit{Berezin transform} $\mathcal{B}^{\varphi_1,\varphi_2}T$ is the function on $\reals^{2d}$ defined by
\begin{equation*}
  \mathcal{B}^{\varphi_1,\varphi_2}T(z)=\inner{T\pi(z)\varphi_{1}}{\pi(z)\varphi_{2}}.
\end{equation*}

We often write just $\mathcal{B}$ and $\mathcal{A}$ when the this does not lead to ambiguity.

\subsection{A Banach space result}
Finally, we will need the following result on the Banach space adjoint $T^*$ of a bounded linear operator $T:X\to Y$ for Banach spaces $X$ and $Y$ \cite[Thm. 4.12]{Rudin:2006ul}. 
\begin{prop} \label{prop:banach}
\
\begin{enumerate}
	\item The range of $T$ is dense if and only if $T^*$ is injective.
	\item The range of $T^*$ is weak* dense if and only if $T$ is injective. 
\end{enumerate}	
\end{prop}

\section{A shift for operators}
In order to introduce the convolution of an operator with a function, we will first need to define a shift for operators. It is well-known that the time-frequency shifts $\pi(z)$ give a projective representation of $\reals^{2d}$ on $\HS$ with respect to the cocycle $c(z,z')=e^{-2 \pi i \omega' \cdot x}$ \cite{Grochenig:2001}, meaning in particular that $\pi(z)\pi(z')=e^{-2 \pi i \omega' \cdot x}\pi(z+z')$. It was noted both by Werner \cite{Werner:1984} and Feichtinger and Kozek \cite{Feichtinger:1998} that one can obtain a unitary representation $\alpha$ of $\reals^{2d}$ on the Hilbert-Schmidt operators $\SC^2$ by defining
\begin{equation*}
  \alpha_z(A)=\pi(z)A\pi(z)^*
\end{equation*}
for $z\in \reals^{2d}$ and $A\in \bo$. It is easily confirmed that $\alpha_z\alpha_{z'}=\alpha_{z+z'}$, and we will informally think of $\alpha$ as a shift or translation of operators. 
\begin{rem}
	Since we defined $\alpha$ by $\alpha_zT=\pi(z)T\pi(z)^*$ for $z=(x,\omega)\in \reals^{2d}$, we can modify $\pi$ by any phase factor without affecting $\alpha$. In particular the family of representations $\pi_{\lambda}(z)=T_{\lambda x}M_{\omega}T_{(1-\lambda)x}$ would all give the same $\alpha$ for $\lambda\in [0,1]$.
\end{rem}

Similarly we define the analogue of $f\mapsto \check{f}$ for an operator $A\in \bo$ by
\begin{equation*}
  \check{A}=PAP,
\end{equation*}
where $P$ is the parity operator. 

The following lemma lists several elementary properties of $\alpha$ and $A\mapsto \check{A}$. The proofs are straightforward and may be found in \cite{Skrettingland:2017}.
\begin{lem}
\label{lem:propsofweyl}
	Let $A\in B(L^2(\reals^{2d}))$ and $z,z'\in \reals^{2d}$. 
	\begin{enumerate}
		\item \label{lem:propsofweyl:isometry} If $T\in \SC^p$ for $1\leq p \leq \infty$, then $\|\alpha_zT\|_{\SC^p}=\|T\|_{\SC^p}$ and $\|\check{T}\|_{\SC^p}=\|T\|_{\SC^p}$.
		\item \label{lem:propsofweyl:sympform} $\alpha_z\pi(z')=e^{2 \pi i\sigma(z,z')}\pi(z')$.
		
		\item \label{lem:propsofweyl:adjoint} $(\alpha_zA)^*=\alpha_z A^*$ and  $\left(\check{A}\right)^*=(A^*)\parcheck$.
		
		\item \label{lem:propsofweyl:parity} $\pi(z)P=P\pi(-z)$, $\widecheck{\pi(z)}=\pi(-z)$  and  $(\alpha_zA)\parcheck = \alpha_{-z}\check{A}$.
	\end{enumerate}
\end{lem}

To show that $\alpha$ is a reasonable definition of translation of operators, Feichtinger and Kozek \cite{Feichtinger:1998} observed that applying $\alpha_z$ to a pseudodifferential operator amounts to a translation of its symbol, a fact that was observed by Kozek already in \cite{ko92}. We make this precise in the following lemma. 
\begin{lem} \label{lem:weylandops}
Let $f\in L^1(\reals^{2d})$, and let $L_f$ be the Weyl transform of $f$. 
\begin{itemize}
	\item $\alpha_z(L_f)=L_{T_zf}$ for $z\in \reals^{2d}$.
	\item $\widecheck{L_f}=L_{\check{f}}$.
	\item $L_f^*=L_{f^*}$.
\end{itemize}	
In particular, if $S\in \mathcal{M}$, then $\alpha_z(S),\check{S}, S^* \in \mathcal{M}$.
\end{lem}

\begin{proof}
	From section \ref{sec:intschrodrep} we know that the twisted Weyl symbol of $L_f$ is $\F_{\sigma}f$, so $L_f=\iint \F_{\sigma}f(z') e^{-i \pi \omega' \cdot x'}\pi(z') \ dz'$ where $z'=(x',\omega')$. Using this representation of $L_f$ will allow us to use the results from lemma \ref{lem:propsofweyl}.
	\begin{enumerate}
		\item From proposition \ref{prop:integralandoperator} and part \eqref{lem:propsofweyl:sympform} of lemma \ref{lem:propsofweyl} we find that
		\begin{align*}
  			\pi(z) L_f \pi(z)^* &= \iint_{\reals^{2d}} \F_{\sigma}f(z') e^{-i \pi \omega ' \cdot x'}\alpha_z(\pi(z')) \ dz' \\
  			&= \iint_{\reals^{2d}} \F_{\sigma}f(z')e^{2 \pi i\sigma(z,z')} e^{-i \pi \omega ' \cdot x'} \pi(z') \ dz' \\
  			&=\iint_{\reals^{2d}} \F_{\sigma}(T_zf)(z') e^{-i \pi \omega ' \cdot x'} \pi(z') \ dz' = L_{T_z f}. 
		\end{align*}
		  We have used that $\F_{\sigma}(T_zf)(z')=\F_{\sigma}f(z')e^{2 \pi i\sigma(z,z')}$.
		  		\item By proposition \ref{prop:integralandoperator} and part \eqref{lem:propsofweyl:parity} of lemma \ref{lem:propsofweyl},
			\begin{align*}
  				PL_fP&=\iint_{\reals^{2d}} \F_{\sigma}f(z') e^{-i \pi \omega ' \cdot x'}P\pi(z')P \ dz' \\
  				&= \iint_{\reals^{2d}} \F_{\sigma}f(z') e^{-i \pi \omega ' \cdot x'}\pi(-z') \ dz' \\
  				&= \iint_{\reals^{2d}} \F_{\sigma}f(-z') e^{-i \pi \omega ' \cdot x'}\pi(z') \ dz' \\
  				&= \iint_{\reals^{2d}} \F_{\sigma}\check{f}(z') e^{-i \pi \omega ' \cdot x'}\pi(z') \ dz'=L_{\check{f}},
			\end{align*}
			where the penultimate step uses $\F_{\sigma}\check{f}=\widecheck{\F_{\sigma}f}$.
		\item Similar to the proofs above. See \cite{Skrettingland:2017} for a full proof.
	\end{enumerate}
	As we have discussed, $\mathcal{M}$ consists of operators with Weyl symbol in $M^1(\reals^{2d})$. We have just shown that if $S$ has Weyl symbol $f$, then $\alpha_zS$, $\check{S}$ and $S^*$ have Weyl symbols $T_zf$, $\check{f}$ and $f^*$, respectively. The last statement now follows from the fact that $M^1(\reals^{2d})$ is closed under these three operations \cite{Jakobsen:2016wx}. 
	\end{proof}

\begin{prop} \label{prop:continuity}
The translation of operators has the following continuity properties:
	\begin{enumerate}
		\item For $1<p<\infty$, the map $z\mapsto \alpha_zT$ is continuous from $\reals^{2d}$ to $\SC^p$ for any fixed $T\in \SC^p$.
		\item The map $z\mapsto \alpha_zT$ is continuous from $\reals^{2d}$ to $K(\HS)$ for any fixed $T\in K(\HS)$.
		\item The map $z\mapsto \alpha_zA$ is weak*-continuous from $\reals^{2d}$ to $\bo$ for any fixed $A\in \bo$.
	\end{enumerate} 
\end{prop}
\begin{proof}
	The first statement follows from Gr{\"u}mm's convergence theorem \cite[Thm. 2.19]{Simon:2010wc}, and the second follows from the first using an approximation argument as shown in \cite{Skrettingland:2017}. 
	The last statement claims that the function $z\mapsto \tr((\alpha_z A)T)$ is continuous for $T\in \tco$. If $\{z_n\}_{\mathbb{N}}$ is a sequence in $\reals^{2d}$ with $z_n \to z\in \reals^{2d}$ and $\{e_n\}_{m\in \mathbb{N}}$ is an orthonormal basis for $\HS$, then 
	\begin{align*}
  			\lim_{n\to \infty} \tr((\alpha_{z_n}A)T^*)&= \lim_{n\to \infty} \sum_{m\in \mathbb{N}} \inner{(\alpha_{z_n}A)T^*\psi_m}{\psi_m} \\
  			&=  \sum_{n\in \mathbb{N}} \lim_{n\to \infty} \inner{(\alpha_{z_n}A)T^*\psi_m}{\psi_m} \\
  			&= \sum_{n\in \mathbb{N}} \inner{(\alpha_{z}A)T^*\psi_m}{\psi_m}=\tr((\alpha_{z}A)T^*),
		\end{align*}
		and one may use part \ref{prop:traceprops:absconv} of proposition \ref{prop:traceprops} and the dominated convergence theorem to justify that the order of the sum and limit may be switched \cite{Skrettingland:2017}.
\end{proof}

\section{Convolutions of operators and functions} \label{sec:convolutions}
Using $\alpha$ we can now define a convolution operation between functions and operators. If $f \in L^1(\reals^{2d})$ and $S \in \tco$ we define the \textit{operator} $f\ast S$ by
\begin{equation*}
  f\ast S := S\ast f = \iint_{\reals^{2d}}f(y)\alpha_y(S) \ dy
\end{equation*}
where the integral is interpreted in the weak and pointwise sense as discussed in section \ref{sec:integration}. By proposition \ref{prop:intistrace} we immediately get that $f\ast S \in \tco$ and $\|f\ast S \|_{\tco}\leq \|f \|_{L^1}\|S\|_{\tco}$.

Werner recognized \cite{Werner:1984} that in order to investigate the convolution of functions with operators, one should consider a corresponding convolution of two operators. For two operators $S,T \in \tco$, Werner defined the \textit{function} $S\ast T$ by 
\begin{equation*}
  S \ast T(z) = \tr(S\alpha_z (\check{T}))
\end{equation*}
for $z\in \reals^{2d}$. That the name "convolution" is apt for these operations is supported by some of their properties proved in this section, and in section \ref{sec:fourier} we will introduce a Fourier transform of operators that interacts with these convolutions in the expected way. 

 The following generalization of Moyal's lemma is \cite[Lem. 3.1]{Werner:1984}. It shows that $S\ast T\in L^1(\reals^d)$ and provides an important formula for its integral.
\begin{lem}
\label{lem:werner}
Let $S,T \in \tco$. The function $z \mapsto \tr(S\alpha_zT)$ for $z\in \reals^{2d}$ is integrable and $\|\tr(S\alpha_zT)\|_{L^1} \leq \|S\|_{\tco} \|T\|_{\tco}$.

Furthermore,
\begin{equation*}
	\iint_{\reals^{2d}} \tr(S\alpha_zT) \ dz = \tr(S)\tr(T).
\end{equation*}
\end{lem}

\begin{proof}	
We start by showing the norm-inequality. First use the singular value decomposition of the operators $S$ and $T$ to write
\begin{align*}
 & S=\sum_{m\in \mathbb{N}} s_m \psi_m\otimes \phi_m
  &&T=\sum_{n\in \mathbb{N}} t_n \eta_n \otimes \xi_n,
\end{align*}
where $\{s_m\}_{m \in \mathbb{N}}$ and $\{t_n\}_{n\in \mathbb{N}}$ are the singular values of $S$ and $T$, respectively, and the sets $\{\psi_m\}_{m\in \mathbb{N}}, \{\phi_m\}_{m\in \mathbb{N}}, \{\eta_n\}_{n\in \mathbb{N}}$ and $\{\xi_n\}_{n\in \mathbb{N}}$ are orthonormal in $L^2(\reals^d)$. Then extend the set $\{\psi_m\}_{m\in \mathbb{N}}$ to an orthonormal basis $\{e_i\}_{i\in \mathbb{N}}$ of $L^2(\reals^d)$. Using this basis to calculate the trace, we find that
	\begin{align} \label{eq:tracestft}
  \tr(S\alpha_zT)&=\sum_{i\in \mathbb{N}} \inner{S\pi(z) T \pi(z)^*e_i}{e_i} \nonumber \\
  &= \sum_{i,m,n \in \mathbb{N}} s_m t_n \inner{\pi(z)^*e_i}{\xi_n} \inner{\pi(z)\eta_n}{\phi_m} \inner{\psi_m}{e_i} \nonumber \\
  &= \sum_{m,n \in \mathbb{N}} s_m t_n \inner{\pi(z)^*\psi_m}{\xi_n} \inner{\pi(z)\eta_n}{\phi_m} \nonumber \\
  &= \sum_{m,n \in \mathbb{N}} s_m t_n V_{\xi_n}\psi_m(z) \overline{V_{\eta_n}\phi_m}(z).
\end{align}

By Moyal's identity, $V_{\xi_n}\psi_m, V_{\eta_n}\phi_m \in L^2(\reals^{2d})$, and so $V_{\xi_n}\psi_m \overline{V_{\eta_n}\phi_m}\in L^1(\reals^{2d})$ by H{\"o}lder's inequality. The following computation shows that the series above converges absolutely in $L^1(\reals^d)$ with the desired norm estimates.  
\begin{align*}
  \| \sum_{m,n\in \mathbb{N}} s_m t_n V_{\xi_n}\psi_m \overline{V_{\eta_n}\phi_m} \|_{L^1} &\leq \sum_{m,n\in \mathbb{N}} s_m t_n \|V_{\xi_n}\psi_m \overline{V_{\eta_n}\phi_m}\|_{L^1} \\
  &\leq \sum_{m,n\in \mathbb{N}} s_m t_n \|V_{\xi_n}\psi_m\|_{L^2} \|V_{\eta_n}\phi_m\|_{L^2} \\
  &= \sum_{m,n\in \mathbb{N}} s_m t_n \|\xi_n\|_{L^2}\|\psi_m\|_{L^2} \|\eta_n\|_{L^2}\|\phi_m\|_{L^2} \\
  &= \sum_{m,n\in \mathbb{N}} s_m t_n = \|S\|_{\tco}\|T\|_{\tco}.
\end{align*}

The equality $\iint_{\reals^{2d}} \tr(S\alpha_zT) \ dz = \tr(S)\tr(T)$ now follows easily from Moyal's identity and equation \ref{eq:tracestft} above. 
\begin{align*}
  \iint_{\reals^{2d}} \tr(S\alpha_zT) \ dz &= \iint_{\reals^{2d}} \sum_{m,n\in \mathbb{N}} s_m t_n V_{\xi_n}\psi_m \overline{V_{\eta_n}\phi_m} \ dz \\
  &=\sum_{m,n\in \mathbb{N}} s_m t_n \iint_{\reals^{2d}} V_{\xi_n}\psi_m \overline{V_{\eta_n}\phi_m} \ dz \\
  &= \sum_{m,n\in \mathbb{N}} s_m t_n \inner{\psi_m}{\phi_m} \inner{\eta_n}{\xi_n} \\
  &= \tr(S) \tr(T),
\end{align*}
where the last equality follows from an easy calculation of $\tr(S)$ and $\tr(T)$.
\end{proof}
\begin{rem}
	 Convolutions of functions with operators could have been defined in the very general setup of a locally compact group $G$ and a strongly continuous projective representation $\{U_z\}_{z\in G}$ on some Hilbert space $\mathcal{H}$. As we have done for $G=\reals^{2d}$ and $U_z=\pi(z)$, one could use proposition \ref{prop:intistrace} to make $\tco(\mathcal{H})$ into a Banach module over $L^1(G)$ by $f\ast T=\int_{G}f(z)U(z)TU(z)^* \ dz$ for $f\in L^1(G)$, $T\in \tco(\mathcal{H})$. Such modules were studied by Bekka in \cite{Bekka:1990}, and to some extent also by Arveson \cite{Arveson:1982} and Graven \cite{Graven:1974}. 
	%However, we will not be able to to get $S\ast T\in L^1(G)$ for $S,T \in \tco(\mathcal{H})$ in this general setup -- the proof of this fact relied on Moyal's identity. There is a Moyal type formula for any square-integrable representation> Franz
	 
	 Another natural extension is therefore to consider a locally compact abelian group $G$ and the Hilbert space $L^2(G)$, with the representation $\pi$ on $G\times \hat{G}$ given by $\pi(x,\omega)f(t)=\omega(t)f(t-x)$ for $x\in G, \omega \in \hat{G}, f\in L^2(G)$. For these representations Moyal's identity is true \cite[p. 10]{Jakobsen:2016wx}, and Kiukas et al. \cite{Kiukas:2012gt} claim that the theory in this section carries over to this more general setting, mutatis mutandis.

\end{rem}
Using duality we can extend the domains of the convolutions introduced above, by allowing one factor to belong to the dual space. For instance, if $h \in L^\infty(\reals^{2d})$ and $S \in \tco$, we define $h\ast S\in \bo$ by $\inner{h \ast S}{T}=\inner{h} {T\ast \check{S}^*}$ for every $T\in \tco$. A standard interpolation argument then gives the following result, since $(L^1(\reals^d),L^{\infty}(\reals^d))_{\theta}=L^{p}$ and $(\SC^1,\SC^{\infty})_{\theta}=\SC^{p}$ with $\frac{1}{p}=1-\theta$ \cite{Toft:2008vr,Bergh:2012up}.

\begin{prop} \label{prop:convschatten}
	Let $1\leq p,q,r \leq \infty$ be such that $\frac{1}{p}+\frac{1}{q}=1+\frac{1}{r}$. If $f\in L^p(\reals^{2d}), g\in L^q(\reals^{2d}), S \in \SC^p$ and $T\in \SC^q$, then the following convolutions may be defined and satisfy the norm estimates
	\begin{align*}
  		\|f\ast T\|_{\SC^r} &\leq \|f\|_{L^p} \|T\|_{\SC^q}, \\
  		\|g\ast S\|_{\SC^r} &\leq \|g\|_{L^q} \|S\|_{\SC^p}, \\
  		\|S\ast T\|_{L^r} &\leq \|S\|_{\SC^p} \|T\|_{\SC^q}.  		
	\end{align*}
	\end{prop}
\begin{rem}
	For $A\in \bo$ and $S\in \tco$, the expression $A\ast S(z)=\tr(A\alpha_z\check{S})$ is still valid \cite{Kiukas:2012gt,Skrettingland:2017}. Also, a simple calculation shows that $A\ast S(z)=\tr(\check{A}\alpha_{-z}S)$ is an equivalent expression -- we will use this expression whenever we find it convenient.
\end{rem}

The next lemma shows that the convolutions interact with translations $\alpha$ and $A\mapsto \check{A}$ in the expected way. We refer to \cite{Skrettingland:2017} for the elementary proof.
\begin{lem}
	\label{lem:tidbits}
	\begin{enumerate}
		Suppose $f \in L^1(\reals^{2d})$ and $S,T \in \tco$.
		\item  $(f \ast S)^*=f^* \ast S^*$  and $(S \ast T)^*=S^* \ast T^*$. 
		\item  $(f \ast S)\parcheck=\check{f}\ast \check{S}$ and $(S \ast T)\parcheck=\check{S}\ast \check{T}$.
		\item \label{lem:tidbits:translate}  $\alpha_z(f \ast S)=(T_z f) \ast S $ and $T_z(S \ast T)=(\alpha_z S) \ast T $. 
	\end{enumerate}
	
\end{lem}

Since the convolutions between operators and functions can produce both operators and functions as output, the associativity of the convolution operations is not trivial. The fact that associativity holds will be exploited frequently later in the text, and we now give a more elaborated version of Werner's proof of this fact \cite{Werner:1984}.
\begin{prop} \label{prop:associative}
The convolution operations in proposition \ref{prop:convschatten} are commutative and associative.
\end{prop}

\begin{proof}
		\textit{Commutativity:} Let $S,T \in \tco$. We find that 
		\begin{align*}
		S \ast T(z)&=\tr (S \alpha_z \check{T}) \\
		&= \tr(S \pi(z) P TP \pi(z)^*) \\
		&= \tr (T  (\alpha_{-z} S)\parcheck) \\
		&= \tr (T \alpha_z  \check{S})= T\ast S (z) 
		\end{align*}
		We have made extensive use of the property $\tr(AB)=\tr(BA)$, and also used part \eqref{lem:propsofweyl:parity} of lemma \ref{lem:propsofweyl}. 		
		
		\textit{Associativity:} The most interesting case is the convolution of three operators. We will need lemma \ref{lem:werner} in addition to some more technical calculations. Let $T_1,T_2,T_3 \in \tco$. To show that $T_1 \ast (T_2 \ast T_3)=(T_1 \ast T_2) \ast T_3$ it will be helpful to assume an arbitrary operator $T_0 \in \tco$. If we can show that the dual space actions $\inner{T_1 \ast (T_2 \ast T_3)}{T_0}=\inner{(T_1 \ast T_2) \ast T_3}{T_0}$ for any $T_0$, we will have shown that the two expressions define the same element in the dual space $B(L^2(\reals^d))$, and therefore the same operator. It will suffice to show that 
		\begin{equation*}
		\tr \left[ T_0 (T_1 \ast (T_2 \ast T_3))\right] = \tr \left[ T_0 ((T_1 \ast T_2) \ast T_3)\right]. 
		\end{equation*}		
		Writing out the left side of the equation and using proposition \ref{prop:intistrace}, we find that 
		\begin{align*}
		\tr \left[ T_0 (T_1 \ast (T_2 \ast T_3))\right] &=\tr \left[T_0 \iint_{\reals^{2d}} \tr(T_2 \alpha_x \widecheck{T_3}) \alpha_x T_1 \ dx 	  \right] \\
		&=   \iint_{\reals^{2d}} \tr\left[T_2 \alpha_x \widecheck{T_3}\right]\tr \left[ (\alpha_x T_1) T_0\right]  \ dx 	   \\
		&=   \iint_{\reals^{2d}} \iint_{\reals^{2d}} \tr\left[(\alpha_x T_1)T_0 \alpha_y( T_3 \alpha_x \widecheck{T_2})\right] \ dy  \ dx.
		\end{align*}
		The last equality uses lemma \ref{lem:werner} to introduce the second integral, and also exploits the commutativity of convolutions to switch the order of $T_2$ and $T_3$. It is a simple exercise to check that $\alpha_y(AB)=(\alpha_yA)(\alpha_y B)$ for operators $A$ and $B$, hence $\alpha_y( T_3 \alpha_x \widecheck{T_2})=(\alpha_yT_3)(\alpha_{x}\alpha_y \widecheck{T_2})$. Using this in our calculation we get that
		{\footnotesize
		\begin{align*}
		\iint_{\reals^{2d}} \iint_{\reals^{2d}} \tr\left[(\alpha_x T_1)T_0 \alpha_y( T_3 \alpha_x \widecheck{T_2})\right] \ dy  \ dx 
		&= \iint_{\reals^{2d}} \iint_{\reals^{2d}} \tr\left[(\alpha_x T_1)T_0 (\alpha_yT_3)(\alpha_{x}\alpha_y \widecheck{T_2}))\right] \ dy  \ dx \\
		&=	\iint_{\reals^{2d}} \iint_{\reals^{2d}} \tr\left[(T_0\alpha_yT_3) (\alpha_{x}\alpha_y \widecheck{T_2})(\alpha_x T_1)\right] \ dy  \ dx.
		\end{align*} \par}
		
		As above, $(\alpha_{x}\alpha_y \widecheck{T_2})(\alpha_x T_1)=\alpha_x((\alpha_y \widecheck{T_2}) T_1)$. We may use Fubini's theorem to change the order of integration, and then invoke the equality in lemma \ref{lem:werner} again to reduce the expression to a form that we recognize as the desired equality.
		\begin{align*}
		\iint_{\reals^{2d}} \iint_{\reals^{2d}} \tr\left[(T_0\alpha_yT_3) \alpha_x((\alpha_y \widecheck{T_2}) T_1)\right] \ dx  \ dy &=	 \iint_{\reals^{2d}} \tr\left[T_0\alpha_yT_3 \right] \tr\left[(\alpha_y \widecheck{T_2}) T_1\right] \ dy  \\
		&= \tr \left[ T_0 ((T_1 \ast T_2) \ast T_3)\right].
		\end{align*}
		As mentioned at the beginning of the proof, the other cases are more elementary, using properties of the weak definition of the integral. 
\end{proof}

Propositions \ref{prop:convschatten} and \ref{prop:associative} imply that $\SC^p$ is a Banach module over $L^1(\reals^{2d})$ for $1\leq p\leq \infty$. In fact, there is a theory of \textit{Banach modules with shifts}, where a Banach module over $L^1(\reals^{2d})$ is constructed in a natural way from a Banach space $X$ with an automorphism $\gamma$ that behaves like a translation on $X$. When we let $X=\SC^p$ for $p<\infty$ and $\gamma=\alpha$, this construction produces the convolution defined by Werner. We refer the interested reader to chapter 3 of \cite{Graven:1974} or \cite{Skrettingland:2017}, but we cite the following consequence that follows directly from \cite[Thm. 3.1.7]{Graven:1974}.
\begin{prop}
	Let $A\in \bo$. The map $z\mapsto \alpha_zA$ is strongly continuous if and only if $A=f\ast T$ for $f\in L^1(\reals^{2d}), T \in \bo$.
\end{prop}

We also note that the compact operators $K(\HS)$ and the uniformly continuous functions vanishing at infinity, $C_0(\reals^{2d})$, are corresponding under convolutions with trace class operators in a sense made precise by the following proposition. A proof may be found in \cite{Skrettingland:2017}.
\begin{prop} \label{prop:compactcontinuous}
	Let $T \in \tco$. If $f\in C_0(\reals^{2d})$ and $S\in K(L^2(\reals))$, then $f \ast T \in K(L^2(\reals))$ and $S\ast T \in C_0(\reals^{2d})$.
\end{prop}
\subsection{Banach space adjoints} \label{sec:adjoints}
We will consider the operation of taking convolutions with a fixed operator $S$, and inspired by the notation for localization operators we introduce the operators $\mathcal{A}_S$ and $\mathcal{B}_S$ by
\begin{align*}
 & \mathcal{A}_S f = f\ast S &&
  \mathcal{B}_S T= T \ast \check{S}^*,
\end{align*}
for a function $f:\reals^{2d}\to \mathbb{C}$ and $T\in \bo$. It was noted by Werner \cite{Werner:1984} in the general case and Bayer and Gr{\"o}chenig \cite{Bayer:2014td} for localization operators that $\mathcal{A}_S$ and $\mathcal{B}_S$ are adjoints of each other, when considered with the appropriate domains.
\begin{thm} \label{thm:adjoints}
Fix $S\in \tco$ and $1\leq p< \infty$. Let $q$ be determined by $\frac{1}{p}+\frac{1}{q}=1$. The Banach space adjoint of $\mathcal{A}_S:L^p(\reals^{2d}) \to \SC^p$ is given by
 	\begin{equation*}
  (\mathcal{A}_S)^*  = \mathcal{B}_{S}, 
	\end{equation*}
	where $\mathcal{B}_{S}:\SC^q\to L^q(\reals^{2d})$. \\
	Similarly, the adjoint of $\mathcal{B}_S:\SC^p \to L^p(\reals^{2d})$ is given by
 	\begin{equation*}
  (\mathcal{B}_S)^* = \mathcal{A}_{S},
\end{equation*}
	where $\mathcal{A}_{S}:L^q(\reals^{2d})\to \SC^q$. 
 \end{thm}
 \begin{proof}
 	If we let the bracket denote duality, then the adjoint of $\mathcal{A}_S$ is determined by 	
\begin{equation*}
  \inner{(\mathcal{A}_S)^* T}{f}=\inner{T}{\mathcal{A}_S f} 
\end{equation*}
 	for any $T \in \SC^q$ and $f \in L^p(\reals^{2d})$. First assume $1\le p<\infty$. One easily checks using the definition of $\mathcal{A}_S$ and $\mathcal{B}_S$ (see \cite{Skrettingland:2017} for a proof) that
 \begin{equation*}
  \inner{\mathcal{B}_S T}{f}=\inner{T}{\mathcal{A}_S f}
\end{equation*}
is true whenever $T\in \tco $ and $f\in L^{1}\cap L^{\infty}$. The general statement then follows, since $\tco$ is a dense subspace of $\SC^q$ and $L^1\cap L^{\infty}$ is a dense subspace of $L^p$. The case $p=\infty$ holds by our definition using duality. The proof that $(\mathcal{B}_S)^* = \mathcal{A}_S$ uses exactly the same argument.
 \end{proof}

\section{Localization operators as convolutions} 
One stated aim of this paper is to relate the theory of localization operators to the convolution operations introduced in the previous section. We make this connection explicit in the next theorem.
\begin{thm} \label{thm:realtionlocberconv}
Fix $\varphi_1,\varphi_2 \in L^2(\reals^d)$, and consider the operators $\varphi_2\otimes \varphi_1$ and $\check{\varphi_1}\otimes \check{\varphi_2}$. Let $T\in \bo$ and let $f$ be a function on $\reals^{2d}$. 

The localization operator $\mathcal{A}_f^{\varphi_1,\varphi_2}$ is given by
\begin{equation*}
	\mathcal{A}_f^{\varphi_1,\varphi_2}=f\ast (\varphi_2\otimes \varphi_1).
\end{equation*}

The Berezin transform of $T$ with windows $\varphi_1$ and $\varphi_2$ is given by
\begin{align*}
  \mathcal{B}^{\varphi_{1},\varphi_{2}} T(z) = T\ast (\check{\varphi_1}\otimes \check{\varphi_2}).
  \end{align*}
\end{thm}

\begin{proof}
	The proof will simply consist of calculating $f\ast (\phi_2\otimes \phi_1)$ and $T\ast (\check{\varphi_1}\otimes \check{\varphi_2})$. First let $\psi\in \HS$ and set $S=\phi_2\otimes \phi_1$. We find that
	\begin{align*}
  		(f\ast S)(\psi)&= \iint_{\reals^{2d}} f(z) (\alpha_{z}S)(\psi) \ dz \\
  		&= \iint_{\reals^{2d}} f(z) \inner{\pi(z)^*\psi}{\varphi_1} \pi(z)\varphi_2  \ dz \\
  		&= \iint_{\reals^{2d}} f(z) V_{\varphi_1}\psi\, \pi(z)\varphi_2  \ dz = A_f^{\varphi_1,\varphi_2}\psi.
	\end{align*}	
Turning to the Berezin transform, let $\{e_n\}_{n\in \mathbb{N}}$ be an orthonormal basis of $L^2(\reals^d)$. Using Parseval's identity, we find that 
	\begin{align*}
	(T\ast (\check{\varphi_1}\otimes \check{\varphi_2}))(z)&= \tr(\check{T} \alpha_{-z}(\check{\varphi_1}\otimes \check{\varphi_2})) \\
	&= \sum_{n\in \mathbb{N}} \inner{\check{T} \pi(-z)\check{\varphi_1}\otimes \check{\varphi_2}\pi(-z)^*e_n}{e_n} \\
	&= \sum_{n\in \mathbb{N}} \inner{\pi(-z)^* e_n}{\widecheck{\varphi_{2}}} \inner{\check{T}\pi(-z)\widecheck{\varphi_{1}}}{e_n} \\
	&=\sum_{n\in \mathbb{N}} \inner{ e_n}{\pi(-z)\widecheck{\varphi_{2}}} \inner{\check{T}\pi(-z)\widecheck{\varphi_{1}}}{e_n} \\
	&= \inner{\check{T}\pi(-z)\widecheck{\varphi_{1}}}{\pi(-z)\widecheck{\varphi_{2}}} \\
	&= \inner{PTP\pi(-z)P\varphi_{1}}{\pi(-z)P\varphi_{2}} \\
	&= \inner{T\pi(z)\varphi_{1}}{\pi(z)\varphi_{2}}= \mathcal{B}^{\varphi_{1},\varphi_{2}} T(z), 
	\end{align*}
	where we have used lemma \ref{lem:propsofweyl} in the last step. 
\end{proof}
By combining the results of proposition \ref{prop:convschatten} with theorem \ref{thm:realtionlocberconv}, we recover well-known Schatten $p$-class results for localization operators and Berezin transforms proved in \cite{Bayer:2014td,Cordero:2003ke}, for instance.
\begin{prop}
	Let $\varphi_1, \varphi_2 \in \HS$ and $1\leq p \leq \infty$.
	\begin{enumerate}
		\item If $a\in L^p(\reals^{2d})$, then $\mathcal{A}_a^{\varphi_1,\varphi_2} \in \SC^p$ with $\|\mathcal{A}_a^{\varphi_1,\varphi_2}\|_{\SC^p} \leq \|a\|_{L^p} \|\varphi_1\|_{L^2}\|\varphi_2\|_{L^2}$.
		\item If $T\in \SC^p$, then $\mathcal{B}^{\varphi_1,\varphi_2} T \in L^p(\reals^{2d})$ with $\|\mathcal{B}^{\varphi_1,\varphi_2}T\|_{L^p} \leq \|T\|_{\SC^p} \|\varphi_1\|_{L^2}\|\varphi_2\|_{L^2}$.
	\end{enumerate}	
\end{prop}

\section{A Fourier transform for operators} \label{sec:fourier}

We will now introduce an analogue of the Fourier transform for a trace class operator $S$. The \emph{Fourier-Wigner transform} $\F_WS$ of $S$ is the function given by
	\begin{equation*}
	\F_W S(z)=e^{-\pi i x \cdot \omega}\tr(\pi(-z)S)
	\end{equation*}
	for $z\in \reals^{2d}$. In the terminology of Werner \cite{Werner:1984,Kiukas:2012gt,Keyl:2015bn} this is the Fourier-Weyl transform, but we follow Folland \cite{Folland:1989tl} and call it the Fourier-Wigner transform.

\begin{lem}
Let $S=\varphi_2 \otimes \varphi_1$ with $\varphi_1, \varphi_2 \in L^2(\reals^d)$. The Fourier-Wigner transform of $S$ is given by 
\begin{equation*}
  \F_W(\varphi_2 \otimes \varphi_1)(z)=A(\varphi_2,\varphi_1)(z),
\end{equation*}
where $A(\varphi_2,\varphi_1)(z)$ is the cross-ambiguity function.
\end{lem}

\begin{proof}
	Let $\{\psi_n\}_{n\in \mathbb{N}}$  be an orthonormal basis for $L^2(\reals^d)$. A calculation using Parseval's identity shows that
	\begin{align*}
\F_W (\varphi_2\otimes \varphi_1)(z) &= e^{-\pi i x \cdot \omega}\tr( \pi(-z)\varphi_2\otimes \varphi_1) \\
&= e^{-\pi i x \cdot \omega} \sum_{n\in \mathbb{N}} \inner{( \pi(-z)\varphi_2\otimes \varphi_1) \psi_n}{\psi_n} \\
&= e^{-\pi i x \cdot \omega} \sum_{n\in \mathbb{N}} \inner{\pi(-z) \varphi_2}{\psi_n}  \inner{\psi_n}{\varphi_1} \\
&= e^{-\pi i x \cdot \omega} \inner{\pi(-z) \varphi_2}{\varphi_1} \\
&= e^{\pi i x \cdot \omega} V_{\varphi_1} \varphi_2(z)=A(\varphi_2,\varphi_1)(z).\qedhere
\end{align*}
\end{proof}

At some later point we are going to need an example of an operator $T$ such that $\F_WT(z)\neq 0$ for any $z \in \reals^{2d}$. We therefore include the following example.

\begin{exmp}[Fourier-Wigner transform]
	\label{exmp:wignertransform}

Consider the Gaussian $\varphi(t)=2^{d/4}e^{-\pi t\cdot t}$ for $t\in \reals^d$ and the operator $S=\varphi \otimes \varphi$. We know that $\F_WS=e^{\pi i x \cdot \omega} V_{\varphi} \varphi(z)$, and then find that $\F_W (\varphi\otimes \varphi)(z)=e^{2 \pi i x\cdot \omega}e^{-\frac{1}{2} \pi z \cdot z}$. \end{exmp}

\begin{prop} \label{prop:wignerisunitary}
The Fourier-Wigner transform extends to a unitary operator $\F_W: \SC^2\to L^2(\reals^{2d})$. This extension is the inverse operator of the integrated Schr{\"o}dinger representation $\rho $, and $\F_W(ST)=\F_W(S) \natural \F_W(T)$ for $S,T \in \SC^2$.
\end{prop}

\begin{proof}
	 By the singular value decomposition, elements of the form $S=\sum_{n=1}^N s_n \psi_n\otimes \phi_n$ are dense in $\SC^2$, where $\{\psi_n\}_{n=1}^N$ and $\{\phi_n\}_{n=1}^N$ are orthonormal sets in $\HS$, $s_n>0$ are the singular values of $S$ and $N\in \mathbb{N}$. By the preceding lemma $\F_W(S)=\sum_{n=1}^N e^{\pi i x \cdot \omega} s_n V_{\phi_n}\psi_n$, and we find that
\begin{align*}
	\|\F_W(S)\|^2_{L^2}&= \inner{\sum_{n=1}^N s_n e^{\pi i x \cdot \omega} V_{\phi_n}\psi_n}{\sum_{m=1}^N s_m e^{\pi i x \cdot \omega} V_{\phi_m}\psi_m}_{L^2} \\
	&= \sum_{m,n=1}^N s_m s_n \inner{ V_{\phi_n}\psi_n}{ V_{\phi_m}\psi_m}_{L^2} \\
	&= \sum_{m,n=1}^N s_ms_n \inner{\psi_n}{\psi_m}_{L^2}\inner{\phi_m}{\phi_n}_{L^2} \\
	&= \sum_{n=1}^N s_n^2 = \|S\|_{\SC^2}^2.
\end{align*}
Hence the Fourier-Wigner transform is an isometry on a dense subspace of $\SC^2$ into $L^2(\reals^{2d})$, and therefore extends to an isometry $\F_W:\SC^2 \to L^2(\reals^{2d})$. 
To show that the extension is the inverse of $\rho$, we consider $T=\varphi_2 \otimes \varphi_1$ for $\varphi_1,\varphi_2 \in \HS$. We have already shown that $\F_W(T)(z)=e^{\pi i x \cdot \omega}V_{\varphi_1}\varphi_2$. If we let $\psi,\phi \in \HS$, we may use the weak formulation of the vector-valued integral defining $\rho$ to calculate
	\begin{align*}
  \inner{\rho(e^{\pi i x \cdot \omega}V_{\varphi_1}\varphi_2)\psi}{\phi}&= \iint_{\reals^{2d}} V_{\varphi_1}\varphi_2 \inner{\pi(z)\psi}{\phi} \ dz \\
  &= \iint_{\reals^{2d}} V_{\varphi_1}\varphi_2 \overline{V_{\psi}\phi} \ dz \\
  &= \inner{\varphi_2}{\phi}\inner{\psi}{\varphi_1},
\end{align*}
where the last equality is Moyal's identity. The last expression clearly equals $\inner{Tf}{g}$. If we denote the identity operator on $\SC^2$ by $\mathcal{I}_{\SC^2}$, we have shown that $\rho \F_WT=\mathcal{I}_{\SC^2}T$. By linearity this equality of operators must hold on the dense subspace of $\SC^2$ spanned by such operators $T$, and therefore $\rho \F_W=\mathcal{I}_{\SC^2}$ by continuity. As $\rho$ is unitary, it has an inverse, which implies that $\F_W$ is a two-sided inverse of $\rho$.

 Since $\rho$ is the inverse of $\F_W$, $\rho\F_W(ST)=ST$. But $\rho \left(\F_W(S) \natural \F_W(T)\right)=\rho \F_W(S)\rho \F_W(T) =ST$ from section \ref{sec:intschrodrep}. The equality $\F_W(ST)=\F_W(S) \natural \F_W(T)$ now follows since $\rho$ is injective. 
\end{proof}
As a simple corollary we obtain the following known result \cite[Prop. 286]{deGosson:2011wq}.
\begin{cor}
  	Let $f \in L^1(\reals^{2d})$ be a function such that the Weyl transform $L_{f}$ is a trace class operator. The trace of $L_{f}$ is given by	
  \begin{equation*}
  \tr(L_{f})=\iint_{\reals^{2d}} f(z) \ dz.
\end{equation*}
  \end{cor}

\begin{proof}
	On the one hand, the previous proposition shows that the $\F_W(L_f)$ is the twisted Weyl symbol of $L_f$, which we know is $\F_{\sigma}f$ from section \ref{sec:intschrodrep}. On the other hand, $\F_W(L_f)(z)=e^{-\pi i x \cdot \omega}\tr(\pi(-z)L_f)$ from the definition of the Fourier Wigner transform. Therefore $\F_{\sigma}f(z)=e^{-\pi i x \cdot \omega}\tr(\pi(-z)L_f)$, and evaluating this at $z=0$ gives the desired equality.
\end{proof}

One can also extend the Fourier-Wigner transform to $S\in \bo$; $\F_W(S)$ will then be a tempered distribution. In fact, we will define $\F_W(S)$ as an element of $M^{\infty}(\reals^{2d})$, which is a subset of the tempered distributions.

\begin{prop}
	Let $A\in \bo$. Then we may define $\F_W(A)\in M^{\infty}(\reals^{2d})$ with $\|\F_W(A)\|_{M^{\infty}}\leq C \|A\|_{B(L^2)}$ for some constant $C$. 
\end{prop}

\begin{proof}
	$M^{\infty}(\reals^{2d})$ is the dual space of $M^1(\reals^{2d})$ \cite{Grochenig:2001}. We therefore define $\F_W(A)$ by requiring that $ \inner{F_W(A)}{g}=\inner{A}{\rho(g)}$ for any $g\in M^1(\reals^{2d})$, where the bracket denotes duality and is antilinear in the second argument. It is a simple exercise to check that this relation holds for $A\in \tco$. Note that we need that $\rho(g)\in \tco$ for this to make sense, and this holds by theorem \ref{thm:integraloperators}.\\
	We now check that this defines $\F_W(A)$ as a bounded functional on $M^1(\reals^{2d})$ and obtain the norm estimate. One can show that $\|\rho(g)\|_{\tco}\leq C \|g\|_{M^1}$ for some constant $C$: If we let $k_g$ be the kernel of $\rho(g)$ as an integral operator, theorem \ref{thm:integraloperators} says that there is a constant $K$ with $\|\rho(g)\|_{\tco}\leq K\|k_g\|_{M^1}$.  Furthermore, $k$ and the Weyl symbol $\sigma$ of $\rho(g)$ are related by operations under which $M^1(\reals^{2d})$ is invariant \cite{he03-1}, and similarly $\sigma$ and $g$ are related by the symplectic Fourier transform, under which $M^1(\reals^{2d})$ is also invariant \cite{felu06}. In conclusion, there is some constant $C$ with $\|\rho(g)\|_{\tco}\leq C \|g\|_{M^1}$. We now find that
	\begin{align*}
  |\inner{F_W(A)}{g}|&=|\inner{A}{\rho(g)}| \\
  &\leq \|A\|_{B(L^2)}\|\rho(g)\|_{\tco} \\
  &\leq C \|A\|_{B(L^2)}\|g\|_{M^1}.
\end{align*}

\end{proof}

The Fourier-Wigner transform shares several properties with the Fourier transform of functions. The next proposition, due to Werner \cite{Werner:1984}, provides an example of this.

\begin{prop}
	\label{prop:convolutionandft}
	Let $f \in L^1(\reals^{2d})$ and $S,T \in \tco$. 
	\begin{enumerate}
		\item $\F_{\sigma}(S \ast T)=\F_W(S)\F_W(T)$.
		\item $\F_W(f \ast S)=\F_{\sigma}(f)\F_W(S)$.
	\end{enumerate}	
\end{prop}

\begin{proof}
	\begin{enumerate}
		\item By definition, $\F_{\sigma}(S \ast T)(z)=\iint_{\reals^{2d}} \tr \left[S\pi(z')\check{T}\pi(z')^*\right] e^{-2 \pi i\sigma(z,z')} \ dz'$. Using part \eqref{lem:propsofweyl:sympform} of lemma \ref{lem:propsofweyl} to get that $e^{-2 \pi i\sigma(z,z')}\pi(z')=\alpha_{-z}\pi(z')$, the integrand may be written in a way that will allow us to use lemma \ref{lem:werner}:
		\begin{align*}
		\tr \left[S\pi(z')\check{T}\pi(z)^*\right] e^{-2 \pi i\sigma(z,z')} &=  \tr \left[Se^{-2 \pi i\sigma(z,z')}\pi(z')\check{T}\pi(z')^*\right] \\
		&= \tr \left[S\pi(-z)\pi(z')\pi(-z)^*\check{T}\pi(z')^*\right]. 
		\end{align*}
		Lemma \ref{lem:werner} then gives that 
		\begin{align*}
		\F_{\sigma}(S \ast T)(z)&=\iint_{\reals^{2d}} \tr \left[S\pi(-z)\alpha_{z'}(\pi(-z)^*\check{T})\right]  \ dz' \\
		&= \tr(S\pi(-z))\tr(\pi(-z)^*\check{T}) \\
		&= \tr(S\pi(-z))\tr(e^{-2 \pi i x \cdot \omega} \pi(z)\check{T}) \\
		&= \tr(e^{- \pi i x \cdot \omega}S\pi(-z))\tr(e^{- \pi i x \cdot \omega} \pi(-z)T) \\
		&= \F_W(S)(z)\F_W(T)(z),
		\end{align*}
		where we have used that $\tr(\pi(z)\check{T})=\tr(\pi(z)PTP)=\tr(P\pi(z)PT)=\tr(\pi(-z)T)$ from part \eqref{lem:propsofweyl:parity} of lemma \ref{lem:propsofweyl}.
		
		\item By proposition \ref{prop:intistrace} we may take the trace inside the integral: 
		\begin{align*}
		\F_W(f \ast S)(z)&= e^{-\pi i x \cdot \omega}\tr \left( \pi(-z)\iint_{\reals^{2d}} f(z')\pi(z')S\pi(z')^* \ dz' \right) \\
		&= e^{-\pi i x \cdot \omega} \iint_{\reals^{2d}} f(z')\tr \left[\pi(-z)\pi(z')S\pi(z')^* \right] \ dz'.
		\end{align*}
		
		A simple manipulation of the integrand using part \eqref{lem:propsofweyl:sympform} of lemma \ref{lem:propsofweyl} yields that $\tr \left[\pi(-z)\pi(z')S\pi(z')^* \right]=e^{-2 \pi i\sigma(z,z')}\tr(\pi(-z)S)$. Inserting this expression into our calculation concludes the proof, since
		\begin{align*}
		\F_W(f \ast S)(z)&=e^{-\pi i x \cdot \omega} \iint_{\reals^{2d}} f(z')e^{2 \pi i\sigma(z,z')}\tr(\pi(-z)S) \ dz' \\
		&= e^{-\pi i x \cdot \omega} \tr(\pi(-z)S) \iint_{\reals^{2d}} f(z')e^{-2 \pi i\sigma(z,z')} \ dz' \\
		&= \F_{\sigma}(f)\F_W(S). \qedhere
		\end{align*} 
	\end{enumerate}
\end{proof}

The previous result is not merely aesthetically pleasing, but will be crucial in several proofs in the rest of this text. This is illustrated by the next two results. The first result gives the Weyl symbol of localization operators, and the second is a generalization of the Riemann-Lebesgue lemma, due to Werner \cite{Werner:1984}.

\begin{cor}
	Let $f \in L^1(\reals^{2d})$ and $S \in \tco$. The twisted Weyl symbol of  $f\ast S$ is the function $\F_{\sigma}(f) \F_W(S)$.
In particular, if $\varphi_1,\varphi_2 \in \HS$, then the twisted Weyl symbol of the localization operator $\mathcal{A}^{\varphi_1,\varphi_2}_f$ is the function $\F_{\sigma}(f)\cdot A(\varphi_2,\varphi_1)$.
\end{cor}  
 
\begin{proof}
 	We know from proposition \ref{prop:wignerisunitary} that $\F_W$ is the inverse operator to the integrated Schr{\"o}dinger representation, and thus returns the twisted Weyl symbol of an operator. From proposition \ref{prop:convolutionandft} we find that $\F_W(f\ast S)=\F_{\sigma}(f)\F_W(S)$.
 	\end{proof} 
\begin{prop}[Riemann-Lebesgue lemma]

	If $S\in \tco$, the Fourier-Wigner transform $\F_W(S)$ is continuous and vanishes at infinity, i.e. $\lim\limits_{|z|\to\infty} |\F_W(z)|=0$.
\end{prop}

\begin{proof}
	\emph{Vanishes at infinity:} $\F_{W}(S\ast S)=\F_W(S)^2$ from proposition \ref{prop:convolutionandft}. By the Riemann-Lebesgue lemma for functions, the left side vanishes at infinity, which clearly implies that $\F_W(S)$ vanishes at infinity.
	
	\emph{Continuity:} Assume that $z_n$ is a sequence converging to some $z$ in $\reals^{2d}$. We need to show that $\F_W(S)(z_n)\to \F_W(S)(z)$. Let $\{\psi_m\}_{m\in \mathbb{N}}$  be an orthonormal basis for $L^2(\reals^d)$ consisting of eigenvectors of $|S|$. By the definition of the trace
	\begin{align*}
	\lim\limits_{n\to \infty} \F_W(S)(z_n)&=\lim\limits_{n\to \infty} \sum_{m\in \mathbb{N}} \inner{e^{-\pi i x_n \cdot \omega_n}\pi(z_n)S\psi_m}{\psi_m} \\
	&= \sum_{m\in \mathbb{N}} \inner{e^{-\pi i x \cdot \omega}\pi(z)S\psi_m}{\psi_m}=\F_W(S)(z),
	\end{align*}
	where we have assumed that the limit can be taken inside the sum for now, and used the strong continuity of $z \mapsto e^{-\pi i x \cdot \omega}\pi(z)$.  By the dominated convergence theorem, we can take the limit inside the sum if $|\inner{\pi(z_n)S\psi_m}{\psi_m}|\leq a_m$ for any $n\in \mathbb{N}$, where $\{a_m\}_{m\in \mathbb{N}}$ is some sequence in $\ell^1$. If $S=U|S|$ is the polar decomposition of $S$, then
     \begin{align*}
  |\inner{\pi(z_n)S\psi_m}{\psi_m}|&= |\inner{\pi(z_n)U|S|\psi_m}{\psi_m}| \\
  &\leq \|\pi(z_n)U|S|\psi_m\|_{L^2} \\
  & \leq \||S|\psi_m\|_{L^2}=s_m,
\end{align*}
where $s_m$ is the $m$'th singular value of $S$. The sequence $\{s_m\}_{m\in \mathbb{N}}$ is summable since $S$ is trace class, so picking $a_m=s_m$ completes the proof.
\end{proof}
 	
 	We end this section by considering a Hausdorff-Young inequality for operators due to Werner \cite{Werner:1984}. The reader should note that there is a misprint in the statement in \cite{Werner:1984}, corrected in the following formulation.
 	
 	\begin{prop}[Hausdorff-Young inequality] \label{prop:hausdorffyoung}
  	Let $1\leq p \leq 2$ and let $q$ be the conjugate exponent determined by $\frac{1}{p}+\frac{1}{q}=1$. If $S\in \SC^p$, then $\F_W(S)\in L^q(\reals^{2d})$  with norm estimate
  	\begin{equation*}
  \|\F_W(S)\|_{L^q} \leq  \|S\|_{\SC^p}.
\end{equation*}
  \end{prop}
  
 \begin{proof}
The result for $p=2$ follows from proposition \ref{prop:wignerisunitary}, where we even have equality of norms. For $p=1$, the result follows from part \eqref{prop:traceprops:abstrace} of proposition \ref{prop:traceprops}, since this proposition gives that
 	\begin{align*}
  |F_W{S}(z)|&= |\tr(\pi(-z)S)| \\
  & \leq \|\pi(-z)\|_{B(L^2)} \|S\|_{\SC^1} = \|S\|_{\SC^1},
\end{align*}
so $\|F_WS\|_{L^{\infty}}\leq \|S\|_{\tco}$. The general result $\|\F_W(S)\|_{L^q} \leq \|S\|_{\SC^p}$ now follows from interpolation as $(\SC^1,\SC^2)_{\theta}=\SC^{p}$ and $(L^{\infty},L^2)_{\theta}=L^q$, where $\frac{1}{p}=1-\frac{\theta}{2}$ and $\frac{1}{q}=\frac{\theta}{2}$ \cite{Toft:2008vr}.
 \end{proof} 
 
 If we pick $S=\psi\otimes \phi$ for $\psi,\phi \in L^2(\reals^d)$ in the Hausdorff-Young inequality, we obtain for $2\leq q <\infty$ that
 \begin{equation*}
  \iint_{\reals^{2d}} |V_{\phi}\psi(z)|^q \ dz \leq \|\psi\|_{L^2}^q\|\phi\|_{L^2}^q.
\end{equation*}
This is Lieb's uncertainty principle (Proposition \ref{prop:lieb}), except for the constant $\left(\frac{2}{q}\right)^d$ that makes Lieb's inequality sharp \cite{li90-1}. Hence we can consider Lieb's uncertainty principle to be a sharp version of the Hausdorff-Young inequality for rank-one operators. As a corollary, we note an extension of Lieb's uncertainty principle to trace class operators.
 
 \begin{cor}
 Let $2\leq q <\infty$. If $S\in \tco$, then 
 \begin{equation*}
  \|\F_W(S)\|_{L^q}\leq  \left(\frac{2}{q}\right)^{d/q}
 \|S\|_{\tco}.
\end{equation*}
	
 \end{cor}
\begin{proof}
	We expand $S=\sum_{m\in \mathbb{N}}s_m \psi_m \otimes \phi_m$ using the singular value decomposition and calculate using Lieb's uncertainty principle:
 \begin{align*}
  \|\F_W(S)\|_{L^q}&=\lim\limits_{n\to \infty} {\left\Vert\ \F_W\left(\sum_{m=1}^ns_m \psi_m \otimes \phi_m\right)\right\Vert\ }_{L^q} \\
  &=\lim\limits_{n\to \infty} \|\sum_{m=1}^ns_m A(\psi_m,\phi_m)\|_{L^q} \\
  &\leq \lim\limits_{n\to \infty} \sum_{m=1}^ns_m \|A(\psi_m,\phi_m)\|_{L^q} \\
  &\leq \lim\limits_{n\to \infty} \sum_{m=1}^ns_m \left(\frac{2}{q}\right)^{d/q}\|\psi_m\|_{L^2}\|\phi_m\|_{L^2} \\
  &= \left(\frac{2}{q}\right)^{d/q} \sum_{m=1}^{\infty} s_m = \left(\frac{2}{q}\right)^{d/q} \|S\|_{\tco}.
\end{align*}
The first step in this calculation uses that $\F_W$ is continuous from $\tco$ to $L^q(\reals^{2d})$. This follows from  proposition \ref{prop:hausdorffyoung}, which says that $\F_W$ is continuous from $\SC^p$ to $L^q(\reals^{2d})$, along with the fact that $\|S\|_{\SC^p}\leq \|S\|_{\tco}$. We also use Lieb's uncertainty principle from proposition \ref{prop:lieb} to bound $\|A(\psi_m,\phi_m)\|_{L^q}$.
\end{proof}

 \section{A generalization of Wiener's Tauberian theorem and density theorems} \label{sec:density}	
 In order to state the main results of this section, we will introduce the notion of \textit{regularity} due to Kiukas et al. \cite{Kiukas:2012gt}. For $1\leq p < \infty$, we say that $g \in L^p(\reals^{2d})$ is \emph{$p$-regular} if the translates $\{T_z g:z\in \reals^{2d}\}$ span a norm dense subspace of $L^p(\reals^{2d})$. Similarly, we say that $S\in \SC^p$ is \emph{$p$-regular} if the translates $\{\alpha_z S:z\in \reals^{2d}\}$ span a norm dense subspace of $\SC^p$. We will often refer to $1$-regularity as \textit{regularity}.
	
	If $g \in L^\infty(\reals^{2d})$ we say that $g$ is \textit{$\infty$-regular} if the translates $\{T_z g:z\in \reals^{2d}\}$ span a weak* dense subspace of $L^\infty(\reals^{2d})$. We say that $S\in \bo$ is \emph{$\infty$-regular} if the translates $\{\alpha_z S:z\in \reals^{2d}\}$ span a norm dense subspace of $K(\HS)$.

\begin{rem}
\begin{enumerate}
	\item It is clear that $\|\cdot\|_{B(L^2)}\leq\|\cdot\|_{\SC^q} \leq\|\cdot\|_{\SC^p} \leq \|\cdot\|_{\SC^1}$ for $1\leq p\leq q<\infty$, and that $\SC^p$ is a dense subspace of $\SC^q$. Thus we get that $p$-regularity implies $q$-regularity for an operator $S$ if $ p \leq q $. 
	This is also true for $q=\infty$, since any Schatten $p$-class is norm dense in $K(\HS)$.
	\item An equivalent definition for an operator $S$ to be $\infty$-regular is that the translates of $S$ span a weak* dense subspace of $\bo$ \cite{Kiukas:2006ty}. We will use both of these formulations.
\end{enumerate}
\end{rem}

We are now ready to state Wiener's famous Tauberian theorem using our newly introduced terminology. The first two of these equivalences were proved already in \cite{Wiener:1932bs} by Wiener, the last one appears for instance as theorem 2.3 in \cite{Edwards:1965ev}.

\begin{thm}[Wiener's Tauberian theorem]
	\label{thm:wienerL1}
	\
	\begin{enumerate}
			\item $f\in L^1(\reals^{2d})$ is regular $\iff$ the set $\{z\in \reals^{2d}: \F_{\sigma}f(z)=0\}$ is empty.
			\item $f\in L^2(\reals^{2d})$ is $2$-regular $\iff$ the set $\{z\in \reals^{2d}: \F_{\sigma}f(z)=0\}$ has Lebesgue measure zero.
			\item $f\in L^\infty(\reals^{2d})$ is $\infty$-regular $\iff$ the set $\{z\in \reals^{2d}: \F_{\sigma}f(z)=0\}$ has dense complement.
		\end{enumerate}
	
\end{thm}
For $1<p<2$, Lev and Olevskii \cite{Olevskii:2011} have shown the existence of two functions in $L^1(\reals)$ with the same set of zeros for the Fourier transform, but where one function is $p$-regular and the other is not. Wiener's Tauberian theorem can therefore not be extended in an obvious way to all values of $1\leq p \leq \infty$. 

The next theorem was proved by Kiukas et al. \cite{Kiukas:2012gt}, expanding an earlier result by Werner \cite{Werner:1984}. This theorem allows us to easily infer our main result, theorem \ref{thm:praktteoremet}, and the proof illustrates the use of Werner's convolutions. We therefore include a proof in appendix A for the benefit of the reader, which essentially is a slightly elaborated version of the proof in \cite{Kiukas:2012gt}.
\begin{thm}
	\label{thm:wernerequiv}
	Let $S\in \tco$, let $1\leq p \leq \infty$, and let $q$ be the conjugate exponent of $p$ determined by $\frac{1}{p}+\frac{1}{q}=1$. The following are equivalent:
	\begin{enumerate}
		\item $S$ is $p$-regular.
		\item If $f \in L^q(\mathbb{R}^{2d})$ and $f\ast S=0$, then $f=0$.
		\item $\SC^p \ast S$ is dense in $L^p(\reals^{2d})$.
		\item If $T \in \SC^q$ and $T \ast S=0$, then $T=0$.
		\item $ L^p(\reals^{2d})\ast S$ is dense in $\SC^p$.
		\item $S \ast S$ is $p$-regular.
		\item For any \textit{regular} $T_0\in \tco$, $T_0\ast S$ is $p$-regular. 
	\end{enumerate}
	The density in points (3) and (5) is in the $p$ norm for $p<\infty$, and weak$^*$ density for $p=\infty$. 
	
	For the case $p=\infty$ we may add two further equivalent statements to the list:
	\begin{enumerate}[(i)]
		\item $K(\HS)\ast S$ is dense in $C_0(\reals^{2d})$ in the $\|\cdot\|_{L^{\infty}}$ norm.
		\item $C_0(\reals^{2d}) \ast S $ is dense in $K(\HS)$ in the operator norm $\|\cdot\|_{B(L^2)}$.
	\end{enumerate}
	
	 	Finally, there exists a $p$-regular operator $S$ for any $1\leq p \leq \infty$.
\end{thm}	
	
Using the previous theorem, Werner \cite{Werner:1984} and Kiukas et al. \cite{Kiukas:2012gt} obtained a version of Wiener's Tauberian theorem for operators.
\begin{thm} \label{thm:genwiener}
		Let $S \in \tco$. 
		\begin{enumerate}
			\item $S$ is regular $\iff$ the set $\{z\in \reals^{2d}: \F_WS(z)=0\}$ is empty.
			\item $S$ is $2$-regular $\iff$ the set $\{z\in \reals^{2d}: \F_WS(z)=0\}$ has Lebesgue measure zero.
			\item $S$ is $\infty$-regular $\iff$ the set $\{z\in \reals^{2d}: \F_WS(z)=0\}$ has dense complement.
		\end{enumerate}
	\end{thm}
	
	\begin{proof}
We prove the first part -- the others follow from the same line of reasoning. By theorem \ref{thm:wernerequiv}, $S$ is regular if and only if $S\ast S$ is regular. By theorem \ref{thm:wienerL1}, $S\ast S$ is regular if and only if $\F_{\sigma}(S\ast S)(z)\neq 0$ for any $z\in \reals^{2d}$. However, proposition  \ref{prop:convolutionandft} gives that $\F_{\sigma}(S\ast S)=(\F_WS)^2$. Thus $\F_{W}(S\ast S)\neq 0$ for any $z\in \reals^{2d}$ if and only if the same holds for $\F_WS$, which completes the proof.  
	\end{proof}	
	For pseudodifferential operators the preceding theorem takes a particularly simple form, and gives a simple procedure for obtaining regular operators.
	\begin{cor}
	Let $S\in \tco$ be the operator on $\HS$ with twisted Weyl symbol $f\in M^1(\reals^{2d})$.
	\begin{enumerate}
			\item $S$ is regular $\iff$ the set $\{z\in \reals^{2d}: f(z)=0\}$ is empty.
			\item $S$ is $2$-regular $\iff$ the set $\{z\in \reals^{2d}: f(z)=0\}$ has Lebesgue measure zero.
			\item $S$ is $\infty$-regular $\iff$ the set $\{z\in \reals^{2d}: f(z)=0\}$ has dense complement.
		\end{enumerate}
\end{cor}
\begin{proof}
	We have shown that the Fourier-Wigner transform of a trace class operator is the twisted Weyl symbol of the operator, so $\F_W(S)=f$.
\end{proof}
\subsection{Consequences of the Cohen-Hewitt factorization theorem}
 The fact that $\SC^p$ is a Banach module over $L^1(\reals^{2d})$ for $1\leq p < \infty$ means that we may use the Cohen-Hewitt theorem, which in this case says that the closed linear span of $L^1(\reals^{2d})\ast \SC^p$ in the $\SC^p$-norm is $L^1(\reals^{2d})\ast \SC^p$ \cite[Thm. 2.1.3]{Graven:1974}. Combining this with theorem \ref{thm:genwiener}, we obtain the following result.
\begin{prop}
	For $1\le p<\infty$, every element of $\SC^p$ can be written as $f\ast S$ for $f\in L^1(\reals^{2d})$, $S\in \SC^p$.
\end{prop}
\begin{proof}
	Theorem \ref{thm:wernerequiv} shows that there exists $S_0\in \SC^p$ with $L^1(\reals^{2d})\ast S_0$ dense in $\SC^p$. Hence the closure of $L^1(\reals^{2d})\ast \SC^p$ is $\SC^p$, and by the Cohen-Hewitt theorem this equals $L^1(\reals^{2d})\ast \SC^p$.
\end{proof}

This result appears to be novel in the specific context introduced by Werner in \cite{Werner:1984}, but is discussed in the general case by Graven \cite{Graven:1974}.
\subsection{Arveson spectrum}
Parts of the results in this section may be formulated using a notion of spectrum for a group of automorphisms on a von Neumann algebra, first formulated by Arveson \cite{Arveson:1974}. Let $X$ be a von Neumann algebra with an automorphism group $\{U_z\}_{z\in \reals^{2d}}$ on $X$, and let $x\in X$. Arveson \cite{Arveson:1982} defined the spectrum $\spct_U(x)$ to be the spectrum of the family of functions $\{z\mapsto \rho(U_zx):\rho\in X_{*}\}$, where $X_{\ast}$ is the predual of $X$ considered as a subspace of the dual space of $X$.

 We will consider $U=\alpha$ and $X=\bo$, and then the predual of $X$ is $\tco$, where $T \in \tco$ acts on $S\in \bo$ by $S\mapsto \tr(T^*S)$ as before. By the spectrum of a function in $f\in L^1(\reals^{2d})$ we will mean the closed support of $\F_{\sigma}f$; with this convention the spectrum of $S\in \bo$ is related to the set of zeros of $\F_W(S)$ in a natural way.

\begin{prop} 
	Let $S\in \tco$. The spectrum $\spct_{\alpha}(S)$ is the closure of the complement of $\{z\in \reals^{2d}: \F_WS(-z)=0\}$.
\end{prop}
\begin{proof}
	By definition, $\spct_{\alpha}(S)$ is the spectrum of the functions $\tr(T^*\alpha_zS)=T^*\ast \check{S}(z)$, i.e. the closure of the complement of the set $Z:=\{z\in \reals^{2d}: \F_{\sigma}(T^*\ast \check{S})=0\ \ \forall \ T\in \tco\}$. By proposition \ref{prop:convolutionandft}, $\F_{\sigma}(T^*\ast \check{S})(z)=\F_W(T^*)(z)\F_W(\check{S})(z)=\F_W(T^*)(z)\F_W(S)(-z)$, hence $\{z\in \reals^{2d}: \F_WS(-z)=0\}$ is a subset of $Z$. To see that $Z=\{z\in \reals^{2d}: \F_WS(-z)=0\}$, note that we have constructed $T\in \tco$ with $\F_W(T^*)(z)\neq0$ for any $z\in \reals^{2d}$ in example \ref{exmp:wignertransform}.
	
\end{proof}
It follows that theorem \ref{thm:wernerequiv} and \ref{thm:genwiener} for $p=\infty$ yield a characterization of those $S\in \tco$ where the Arveson spectrum is all of $\reals^{2d}$.

\subsection{Tauberian theorems for localization operators}
In order to apply these results to localization operators, we pick $S=\varphi_2 \otimes \varphi_1$ for two windows $\varphi_1, \varphi_2 \in \HS$, and use theorem \ref{thm:realtionlocberconv} to formulate theorem \ref{thm:wernerequiv} in the terminology of the Berezin transform and localization operators. 
\begin{thm}
	\label{thm:wernerequivloc}
	Fix two windows $\varphi_1, \varphi_2 \in \HS$ for $\mathcal{A}$ and $\mathcal{B}$, let $1\leq p \leq \infty$ and let $q$ be the conjugate exponent of $p$, i.e. $\frac{1}{p}+\frac{1}{q}=1$. The following are equivalent:
	\begin{enumerate}
		\item The operator $\varphi_2\otimes \varphi_1$ is $p$-regular.
		\item $\mathcal{A}$ is injective on $L^q(\reals^{2d})$.
		\item The set $\{\mathcal{B}T : T \in \SC^p\}$ is dense in $L^p(\reals^{2d})$.
		\item $\mathcal{B}$ is injective on $\SC^q$.
		\item The set $\{\mathcal{A}f : f \in L^p(\reals^{2d})\}$ is dense in $\SC^p$.
		\item $\mathcal{B}(\varphi_2 \otimes \varphi_1)$ is $p$-regular.
		\item For any \textit{regular} $T_0\in \tco$, $\mathcal{B}T$ is $p$-regular. 
	\end{enumerate}
	The density in points (1), (3) and (5) is in the $p$ norm for $p<\infty$, and weak$^*$ density for $p=\infty$.
	\end{thm}	
If we apply theorem \ref{thm:genwiener} to localization operators, we obtain a characterization of those windows $\varphi_1,\varphi_2 \in \HS$ that give localization operators $\mathcal{A}^{\varphi_1,\varphi_2}$ with dense range in $\tco$, $\SC^2$ and $\bo$. This improves results by Bayer and Gr{\"o}chenig \cite{Bayer:2014td}.
\begin{thm} \label{thm:praktteoremet}
	Fix two windows $\varphi_1, \varphi_2\in \HS$.
	
	\begin{enumerate}
		\item The set $\{\mathcal{A}^{\varphi_1,\varphi_2}_f : f \in L^1(\reals^{2d})\}$ is dense in $\SC^1$ $\iff$ the set $\{z\in \reals^{2d}: A(\varphi_2,\varphi_1)(z)=0\}$ is empty.
		\item The set $\{\mathcal{A}^{\varphi_1,\varphi_2}_f : f \in L^2(\reals^{2d})\}$ is dense in $\SC^2$ $\iff$ the set $\{z\in \reals^{2d}: A(\varphi_2,\varphi_1)(z)=0\}$ has zero Lebesgue measure.
		\item The set $\{\mathcal{A}^{\varphi_1,\varphi_2}_f : f \in L^{\infty}(\reals^{2d})\}$ is weak* dense in $\bo$ $\iff$ the set $\{z\in \reals^{2d}: A(\varphi_2,\varphi_1)(z)=0\}$ has dense complement.
	\end{enumerate}
\end{thm}

\begin{proof}
The density statements in this theorem are part (5) of theorem \ref{thm:wernerequivloc} for $p=1,2, \infty$, and therefore equivalent to $S$ being  $p$-regular. Furthermore, theorem \ref{thm:genwiener} relates the condition that $S$ is $p$-regular for these three values of $p$ to the set of zeros of $\F_W(S)=A(\varphi_2,\varphi_1)$. 
\end{proof}

\begin{rem}
\begin{itemize}
	\item 	Of course, these statements are also equivalent to the other statements in theorem \ref{thm:wernerequivloc} for $p=1$, $p=2$ and $p=\infty$, respectively.
	\item To the authors' knowledge, there are no results characterizing the functions $\varphi_1,\varphi_2\in L^2(\reals^{d})$ such that the set of zeros of $A(\varphi_2,\varphi_1)$ satisfy the conditions in theorem \ref{thm:praktteoremet}. This question is therefore a possible topic for further research.  
\end{itemize}

\end{rem}

The equivalence in theorem \ref{thm:praktteoremet} for $p=2$ was proved by Bayer and Gr{\"o}chenig in \cite{Bayer:2014td} using different methods. For $p=1,\infty$ they were able to prove sufficient conditions on the zero set of the cross-ambiguity function $A(\varphi_2,\varphi_1)$ for the localization operators to have dense range, but not necessity as we have done in theorem \ref{thm:praktteoremet}. For $p\neq 1,2, \infty$ our approach does not yield equivalences. We may, however, reprove the results in \cite{Bayer:2014td}, since $p$-regularity implies $p'$-regularity for $p\leq p'$. 

\begin{cor} \label{cor:densitylp}
	Fix two windows $\varphi_1, \varphi_2\in \HS$. 
	\begin{enumerate}
		\item Assume $1\leq p < 2$. If the set $\{z\in \reals^{2d}: A(\varphi_2,\varphi_1)(z)=0\}$ is empty, then the set $\{\mathcal{A}^{\varphi_1,\varphi_2}_f : f \in L^p(\reals^{2d})\}$ is norm dense in $\SC^p$.
		\item Assume $2\leq p < \infty$. If the set $\{z\in \reals^{2d}: A(\varphi_2,\varphi_1)(z)=0\}$ has Lebesgue measure zero, then the set $\{\mathcal{A}^{\varphi_1,\varphi_2}_f : f \in L^p(\reals^{2d})\}$ is norm dense in $\SC^p$.
	\end{enumerate}
\end{cor}

\begin{proof}
	If the set $\{z\in \reals^{2d}: A(\varphi_2,\varphi_1)(z)=0\}$ is empty, then the operator $S=\varphi_2 \otimes \varphi_1$ is regular by theorem \ref{thm:praktteoremet} and the following remark. As noted, this implies that $S$ is $p$-regular for $1\leq p < \infty$. By theorem \ref{thm:wernerequivloc}, this implies that $\{\mathcal{A}^{\varphi_1,\varphi_2}_f : f \in L^p(\reals^{2d})\}$ is dense in $\SC^p$.
	Exactly the same argument works for $2\leq p < \infty$, using part (2) of theorem \ref{thm:praktteoremet}. 
\end{proof}

\section{Tauberian theorems for modulation spaces}
The theory of localization operators has been extensively studied in the framework of modulation spaces \cite{Cordero:2003ke,Bayer:2014td,Boggiatto:2004}, and we would like to consider the convolutions in this framework. It is natural to look for a class of operators $\mathcal{M}$ such that $T\ast S\in M^1(\reals^{2d})$ for any $T\in \tco$ and $S\in\mathcal{M}$. It turns out that this is true when $\mathcal{M}$ is the class of operators with Weyl symbol in $M^1(\reals^{2d})$. There are other reasons for considering $\mathcal{M}$ too. Any class of pseudodifferential operators suitable for the theory of convolutions should be closed under $\alpha$, and hence by lemma \ref{lem:weylandops} the associated class of symbols should be closed under translation. We are therefore led to the class $\mathcal{M}$, as $M^1(\reals^{d})$ is the smallest Banach space in $L^1(\reals^{d})$ isometrically invariant under $M_{\omega}$ and $T_x$ \cite{Heil:2008uo}.

In particular, $\mathcal{M}$ contains operators of the form $\varphi_2\otimes \varphi_1$ for $\varphi_1, \varphi_2 \in M^1(\reals^{d})$. Theorem \ref{thm:integraloperators} shows how any operator in $\mathcal{M}$ may be constructed from such simple operators. This is exploited when proving the next two results, which are generalizations of \cite[Thm. 3.4]{Bayer:2014td} and \cite[Thm. 3.6]{Bayer:2014td}. Complete proofs may be found in \cite{Skrettingland:2017}; in this paper we just show how to use theorem \ref{thm:integraloperators} to extend the results in \cite{Bayer:2014td}.

\begin{thm}
\label{thm:locandmod}
Let $1\leq p \leq \infty$, $k\in M^1(\reals^{2d})$ and let $S\in \mathcal{M}$ be the trace class operator given on $L^2(\reals^d)$ by $S\psi(s)=\int_{\reals^d} k(s,t) \psi(t) \ dt $. If $f\in M^{p,\infty}(\reals^{2d})$, the function $f\ast S$ lies in $\SC^p$ and $\|f\ast S\|_{\SC^p}\leq C \|k\|_{M^1} \|f\|_{M^{p,\infty}}$ for some constant $C$ independent of $S$ and $f$.
\end{thm}

\begin{proof}
	Cordero and Gr{\"o}chenig \cite[Thm. 3.4]{Cordero:2003ke} proved that if $S=\varphi_2\otimes \varphi_1$ for $\varphi_1,\varphi_2 \in M^1(\reals^d)$, then $\|f\ast S\|_{\SC^p}\leq C' \|f\|_{M^{p,\infty}} \|\varphi_1\|_{M^1}\|\varphi_2\|_{M^1}$ for some constant $C'$. In order to extend the result to $S=A_k$ for $k\in M^1(\reals^{2d})$, we use the Wilson basis $\{w_m\}_{m\in \mathbb{N}}$ of $\HS$ to write $S=\sum_{m,n\in \mathbb{N}} \inner{k}{W_{mn}}w_m\otimes w_n$, as is allowed by theorem  \ref{thm:integraloperators}. Then
	\begin{align*}
  f\ast S &= f\ast \sum_{m,n\in \mathbb{N}} \inner{k}{W_{mn}}w_m\otimes w_n \\
  &= \sum_{m,n\in \mathbb{N}} \inner{k}{W_{mn}}f \ast \left(w_m\otimes w_n\right).
\end{align*}
Now take the $\SC^{p}$ norm of this expression to find that
\begin{align*}
  \|f\ast S\|_{\SC^{p}}&\leq \sum_{m,n\in \mathbb{N}} |\inner{k}{W_{mn}}| \cdot \|f \ast \left(w_m\otimes w_n\right)\|_{\SC^{p}} \\
  &\leq C_2 \|f\|_{M^{p,\infty}} \sum_{m,n\in \mathbb{N}} |\inner{k}{W_{mn}}| \\
  &\leq C \|f\|_{M^{p,\infty}} \|k\|_{M^1},
\end{align*}
where $C_2$ appears from the result of Gr{\"o}chenig and Cordero and the fact that the $M^1$ norm of the Wilson basis is bounded.
\end{proof}

\begin{thm}
\label{thm:isinfeichtinger}
Let $S\in \mathcal{M}$ be the trace class operator given on $L^2(\reals^d)$ by $S\psi(s)=\int_{\reals^d} k(s,t) \psi(t) \ dt $ with $k\in M^1(\reals^{2d})$. For $T\in \SC^1$, the function 
\begin{equation*}
	T\ast S(z) = \tr(T\pi(z) \check{S} \pi(z)^*)
\end{equation*}
lies in $M^{1}(\reals^{2d})$ and $\|T\ast S\|_{M^{1}}\leq C \|k\|_{M^1} \|T\|_{\SC^1}$ for some constant $C$ independent of $S$ and $T$.
\end{thm}

\begin{rem}
\begin{enumerate}
	\item The theorem is not true for any $S\in \tco$. If $S=T=\psi\otimes \psi$ for $\psi\in L^2(\reals)$, then  $\tr(T\alpha_zS)=|V_{\psi}\psi|^2$. Benedetto and Pfander have constructed an example of a $\psi \in L^2(\reals)$ such that $|V_{\psi}\psi|^2 \notin M^1(\reals^2)$ \cite[Remark 4.6]{Pfander:2006}.
	\item A version of theorem \ref{thm:isinfeichtinger} for $T$ in other Schatten $p$-classes was proved in \cite{Bayer:2014td}. However, the proof uses that $\|\sigma\|_{M^{\infty}}\leq \|T\|_{B(L^2)}$, where $\sigma$ is the Weyl symbol of $T$ as an operator $T:M^1(\reals^d)\to M^{\infty}(\reals^{d})$. This inequality is not true in general; in fact the opposite inequality is true \cite[Thm. 3.1]{Cordero:2003ke}. We therefore settle for the special case $T\in \SC^1$.
	\end{enumerate}
\end{rem}

As mentioned, $\varphi_2\otimes \varphi_1\in \mathcal{M}$ for $\varphi_1, \varphi_2 \in M^1(\reals^{d})$. As a special case of corollary \ref{cor:densitylp} we therefore obtain theorem 5.4 in \cite{Bayer:2014td}. The proof consists merely of noting that $L^p \subset M^{p,\infty}$, so the subsets that we claim are dense are larger than those in corollary \ref{cor:densitylp}. In this sense the following result is weaker than corollary \ref{cor:densitylp}, and the main point of interest is the previously proved fact that symbols in the large space $M^{p,\infty}$ actually give operators in $\SC^p$. 

\begin{cor}
	Fix two windows $\varphi_1, \varphi_2\in M^1(\reals^d)$. 
	\begin{enumerate}
		\item Assume $1\leq p < 2$. If the set $\{z\in \reals^{2d}: A(\varphi_2,\varphi_1)(z)=0\}$ is empty, then the set $\{\mathcal{A}^{\varphi_1,\varphi_2}_f : f \in M^{p,\infty}(\reals^{2d})\}$ is norm dense in $\SC^p$.
		\item Assume $2\leq p < \infty$. If the set $\{z\in \reals^{2d}: A(\varphi_2,\varphi_1)(z)=0\}$ has Lebesgue measure zero, then the set $\{\mathcal{A}^{\varphi_1,\varphi_2}_f : f \in M^{p,\infty}(\reals^{2d})\}$ is norm dense in $\SC^p$.
	\end{enumerate}

\end{cor}

  \section*{Acknowledgement}
We thank the reviewer who brought a mistake in the original version of proposition \ref{prop:hausdorffyoung} to our attention. 
  \appendix
  \section{Proof of theorem \ref{thm:wernerequiv}}
  We begin by showing the existence of a $p$-regular operator. Since a regular operator is $p$-regular for $1\leq p\leq \infty$, it is sufficient to find a \textit{regular} operator $S \in \tco$. Now consider the Gaussian $\varphi$ from example \ref{exmp:wignertransform}; we will prove that $S=\varphi\otimes \varphi$ satisfies part (6) of the theorem, so when the proof of the theorem is complete we know that it is in fact regular. The reason for starting with this is that we will need an operator satisfying (6) during the proof. Proposition \ref{prop:convolutionandft} gives us that $\F_{\sigma}(S\ast S)=\F_W(S) \F_W(S)$. By example \ref{exmp:wignertransform}, $\F_W(S)$ has no zeros, thus the same is true for  $\F_{\sigma}(S\ast S)$. Theorem \ref{thm:wienerL1} then states that $S\ast S$ is regular. 
 
(2) $\iff$ (3): 
	First assume that $1\le p< \infty$. In the notation of section \ref{sec:adjoints}, statement (3) says that $\mathcal{B}_{\check{S}^*}:\SC^p \to L^p(\reals^{2d})$ has dense range. From theorem \ref{thm:adjoints} we know that the Banach space adjoint of $\mathcal{B}_{\check{S}^*}$ is $\mathcal{A}_{\check{S}^*}$, and part (1) of proposition \ref{prop:banach} states that the range of $\mathcal{B}_{\check{S}^*}$ is dense if and only if $\mathcal{A}_{\check{S}^*}$ is injective. Furthermore, the injectivity of $\mathcal{A}_{\check{S}^*}$ is equivalent to the injectivity of $\mathcal{A}_{S}$, since $f \ast \check{S}^*=0$ implies that $((f\ast \check{S}^*)\parcheck)^*=\check{f}^* \ast S=0$ by lemma \ref{lem:tidbits}. This proves the equivalence. 

	For $p=\infty$ we need to use part of (2) of proposition \ref{prop:banach}, and that $(\mathcal{A}_S)^*=\mathcal{B}_{S}$ by theorem \ref{thm:adjoints}. Otherwise, the proof is the same. 
	
	(4) $\iff$ (5):
	Follows from the same line of reasoning as above, with the roles of $\mathcal{A}$ and $\mathcal{B}$ switched.  
	
	(2) $\implies$ (4):
	 Assume that $T \ast S=0$ for some $T \in \SC^q$. Taking the convolution with an arbitrary $A\in \tco$ from the left on both sides of this equality, we find by associativity that $(A\ast T)\ast S=0$. But $A\ast T \in L^q(\reals^{2d})$, so since we are assuming (2) we get that $A\ast T=0$ for any $A \in \tco$. We will use this to show that $T=0$.
	
	As we have remarked earlier, the expression $A\ast T(z)=\tr(\check{A} \alpha_{-z}T)$ is valid, even for $q=\infty$. If we let $z=0$, we have that $A\ast T(0)=\tr(\check{A}T)=0$ for any $A\in \tco$. If we consider $T$ as an element of the dual space $(\tco)^*=\bo$, this says that $\inner{T}{\check{A}^*}=0$ for any $A\in \tco$, where $\inner{\cdot}{\cdot}$ is a duality bracket. This is clearly equivalent to $\inner{T}{A}=0$ for any $A\in \tco$, thus $T=0$.
	
(1) $\iff$ (4):
	For $p<\infty$, we will use that $\SC^q$ is the dual space of $\SC^p$. The subspace spanned by  $\{\alpha_zS:z\in \reals^{2d} \}$ is dense in $\SC^p$ if and only if $\inner{T}{\alpha_zS} =0$ for every $z \in \reals^{2d}$ implies that $T=0$, where the bracket denotes duality and $T \in \SC^q$. Since $T=0$ exactly when $\check{T}^*=0$, we may equivalently phrase this condition as $\inner{\check{T}^*}{\alpha_zS} =0$ for every $z \in \reals^{2d}$ implies that $T=0$. 
	
	 It remains to show that this last statement is equivalent to (4), which may be achieved by showing that $T\ast S=0$ is equivalent to $\inner{\check{T}^*}{\alpha_zS}=0$ for any $z\in \reals^{2d}$.
	By definition $T \ast S(z)=\tr(\check{T}\alpha_{-z}S)$, which we may write in terms of the duality bracket as $T \ast S(z)=\overline{\inner{\check{T}^*}{\alpha_{-z}S}}$. Clearly the left side is zero for all $z\in \reals^{2d}$ if and only if the right side is, which is what we wanted to prove. 
	
	For $p=\infty$ the same argument works, using the fact that $\tco$ is the dual space of $K(\HS)$.
		
	(6) $\implies$ (3):
	By lemma \ref{lem:tidbits}, $T_z(S\ast S)=\alpha_z(S)\ast S$. Therefore (6) gives that the set $\{\alpha_z(S)\ast S:z \in \reals^{2d}\}$ is dense in $L^p(\reals^{2d})$. Since $\alpha_z(S) \in \SC^p$ for any $z \in \reals^{2d}$, this implies in particular the density of $\SC^p \ast S$.
	
	(7) $\implies$ (3):
	As above, $T_z(T_0\ast S) = \alpha_z(T_0) \ast S$, and since $T_0\ast S$ is assumed to be $p$-regular, the set $\{\alpha_z(T_0)\ast S:z \in \reals^{2d}\}$ is dense in $L^p(\reals^{2d})$. But $\{\alpha_z(T_0):z \in \reals^{2d}\} \subset \tco \subset \SC^p$, which proves the statement.
	
	(4) $\implies$ (2):
	Assume that $f\ast S=0$ for $f\in L^q(\reals^{2d})$; we want to show that $f=0$. Taking the convolution with an arbitrary $S'\in \tco$ from the left and using associativity, $(S'\ast f) \ast S =0$, which by (4) implies that $S' \ast f=0$ for any $S' \in \tco$. 
	
	Now let $T$ be the operator $\varphi \otimes \varphi$ discussed at the very start of the proof. We know that $T$ satisfies (6) for $p=1$, and we have shown that $T$ then satisfies (3), i.e. $T \ast \tco$ is a dense subset of $L^1(\reals^{2d})$. Since $f\ast S'=S'\ast f=0$ for any $S'\in \tco$, we must also have $(T\ast S')\ast f=0$ for any $S'\in \tco$. In other words, $f\ast g=0$ for $g$ in the dense subset $T \ast \tco \subset L^1(\reals^{2d})$. Since the convolutions are continuous in both arguments, this shows that $f\ast L^1(\reals^{2d})=0$ and therefore that $f=0$. 
	
	(3) $\implies$ (6):
	Assume first that $p<\infty$.
	Pick an $f\in L^p(\reals^{2d})$ and an $\epsilon >0$. We need to approximate $f$ in the $L^p$-norm by a finite linear combination of elements of the form $T_{z}(S\ast S)=\alpha_z(S)\ast S$. We have proved $(3) \implies (2) \implies (4) \implies (1)$, so the elements $\{\alpha_zS:z \in \reals^{2d}\}$ span a dense subset of $\SC^p$. 
	
	Since (3) holds, we pick an operator $T \in \SC^p$ such that $\|T\ast S-f\|_{L^p} < \frac{\epsilon}{2}$. As we know that (1) holds, we then pick $c_i \in \mathbb{C}$ and $z_i\in \reals^{2d}$ for $i=1,2,...,n$ such that $\|T-\sum_{i=1}^n c_i \alpha_{z_i}S\|_p< \frac{\epsilon}{2\|S\|_1}$. An estimate now shows that
	\begin{align*}
 \|\sum_{i=1}^n c_i \alpha_{z_i}(S)\ast S-f\|_{L^p} &\leq \|\left(\sum_{i=1}^n c_i \alpha_{z_i}S-T\right)\ast S\|_{L^p} + \|T\ast S-f\|_{L^p} \\
  &< \frac{\epsilon}{2\|S\|_{\tco}}\|S\|_{\tco}+\frac{\epsilon}{2}= \epsilon,
\end{align*}
 where we have used proposition \ref{prop:convschatten} to estimate the norm of a convolution. 

 If $p=\infty$, the same basic idea applies. First approximate $f$ by $T\ast S$ in the weak* topology, then approximate $T$ by a finite linear combination of translates of $S$. We leave to the reader the trivial reformulation of the proof in terms of open sets. 

 (3) $\implies$ (7):
 	 Let $T_0\in \tco$ be regular; as noted before, $T_0$ is then also $p$-regular. The key parts of the previous argument was to first use (3) to approximate $f$ by $T\ast S$ for some $T\in \tco$, and then use the $p$-regularity of $S$ to approximate $T$ by a finite linear combination of translates of $S$. Exactly the same argument works in this case, except that we need to approximate $T$ with a finite linear combination of translates of $T_0$ instead of $S$. We leave the details to the reader. 
 	 
 (ii) $\implies$ (i):
 Following \cite{Kiukas:2012gt} we start by showing that (ii)$\implies$(i) when $S_0$ is regular. Fix $f\in C_0(\reals^{2d})$ and $\epsilon > 0$. Then pick $g \in L^1(\reals^{2d})$ such that $\|g \ast f-f\|_{\infty}<\frac{\epsilon}{2}$, which is possible by the existence of approximate identities in $L^1(\reals^{2d})$ \cite[Prop. 2.44]{Folland:2015hg}. Since we are assuming that $S_0$ is regular, we know by (3) that there is a $T\in \tco$ with $\|S_0\ast T-g \|_1<\frac{\epsilon}{2 \|f\|_{\infty}}$. From proposition \ref{prop:compactcontinuous} the operator $T\ast f$ is compact, and an estimate now shows that $S_0\ast (T\ast f)$ approximates $f$:
 \begin{align*}
  \|S_0\ast T \ast f-f\|_{L^{\infty}} &\leq \|S_0\ast T \ast f - g\ast f\|_{L^{\infty}}+\|g\ast f - f\|_{L^{\infty}} \\
  &<\|S_0\ast T-g\|_{L^1} \|f\|_{L^{\infty}}+\frac{\epsilon}{2} < \epsilon.
\end{align*}

Armed with this knowledge we now prove that (ii) $\implies$ (i) for non-regular $S\in \tco$, so assume that $C_0(\reals^{2d}) \ast S $ is dense in $K(\HS)$. We need to prove that $K(\HS)\ast S$ is dense in $C_0(\reals^{2d})$. If $S_0$ is some regular operator, then the set $S_0\ast C_0(\reals^{2d}) \ast S=\{S_0 \ast f \ast S : f \in C_0(\reals^{2d})\}$ is a subset of $K(\HS)\ast S$, and it will be enough to show that this smaller set is dense. Since we assume (ii) we know that $C_0(\reals^{2d}) \ast S$ is dense in $K(\HS)$. We also know that $S_0 \ast K(\HS)$ is dense in $C_0(\reals^{2d})$ from the first part of the argument, and if we combine these two density results with the continuity of the convolutions, we get that $S_0\ast C_0(\reals^{2d}) \ast S$ must be a dense subset of $C_0(\reals^{2d})$.

(i) $\implies$ (ii): 
We will just show that (ii) holds for a regular operator $S_0$. The proof is then completed in the same way as (ii) $\implies$ (i). Let $T\in K(\HS)$ and $\epsilon>0$; we will use three density results to find $f\in C_0(\reals^{2d})$ with $\|T-f\ast S_0\|_{B(L^2)}<\epsilon$. Firstly, we use that $T$ is compact to find a finite rank operator $A$ with $\|T-A\|_{B(L^2)}<\frac{\epsilon}{3}$. Secondly, since $A$ is finite rank it is in particular trace class, so by (3) we may find $g\in L^1(\reals^{2d})$ such that $\|A-g\ast S_0\|_{B(L^2)}<\frac{\epsilon}{3}$. Here we have used that $\|\cdot\|_{B(L^2)}\leq \|\cdot\|_{\tco}$. Finally the continuous functions with compact support are dense in $L^1$, so we can pick $f\in C_0(\reals^{2d})$ such that $\|g-f\|_{L^1}\leq \frac{\epsilon}{3\|S_0\|_{B(L^2)}}$. We claim that $\|T-S_0\ast f\|_{B(L^2)}< \epsilon$, which would conclude the proof. By the triangle inequality
\begin{align*}
  \|T-S_0\ast f\|_{B(L^2)}&\leq \|T-A\|_{B(L^2)} + \|A-S_0\ast g\|_{B(L^2)} + \|S_0\ast g-S_0 \ast f\|_{B(L^2)} \\
  &< \frac{\epsilon}{3}+\frac{\epsilon}{3}+\|S_0\|_{B(L^2)} \|f-g\|_{L^1} < \epsilon.
\end{align*}

(4) $\iff$ (ii) for $p=\infty$: 
This part follows from the same kind of argument as (2) $\iff$ (3) by using proposition \ref{prop:banach} with the Banach spaces $K(\HS)$ and $C_0(\reals^{2d})$. Similar to that argument we get that $C_0(\reals^{2d})\ast S$ is dense in $K(\HS)$ if and only if the map $T\mapsto T\ast S$ is injective from $\tco$ to $C_0(\reals^{2d})^*$. The first statement is clearly (ii), and the last statement is almost (4) when $p=\infty$, except that the codomain is $C_0(\reals^{2d})^*$ rather than $L^1(\reals^{2d})$. However, it should be clear that this is of no importance when determining whether the mapping is injective since $L^1(\reals^{2d})$ may be identified with a subset of $C_0(\reals^{2d})^*$.

\end{document}